\documentclass[sn-mathphys-num]{sn-jnl}% Math and Physical Sciences Numbered Reference Style 
\usepackage{amssymb}
\usepackage{amsmath}
\usepackage{graphicx}%
\usepackage{multirow}%
\usepackage{amsmath,amssymb,amsfonts}%
\usepackage{amsthm}%
\usepackage{mathrsfs}%
\usepackage[title]{appendix}%
\usepackage{xcolor}%
\usepackage{textcomp}%
\usepackage{manyfoot}%
\usepackage{booktabs}%
\usepackage{algorithm}%
\usepackage{algorithmicx}%
\usepackage{algpseudocode}%
\usepackage{listings}%
\newtheorem{theorem}{Theorem}%  meant for continuous numbers
\newtheorem{remark}{Remark}%

\begin{document}

\title[Measurability and continuity of parametric low-rank approximation]{Measurability and continuity of parametric low-rank approximation in Hilbert spaces: linear operators and random variables}

%%=============================================================%%
%% GivenName	-> \fnm{Joergen W.}
%% Particle	-> \spfx{van der} -> surname prefix
%% FamilyName	-> \sur{Ploeg}
%% Suffix	-> \sfx{IV}
%% \author*[1,2]{\fnm{Joergen W.} \spfx{van der} \sur{Ploeg} 
%%  \sfx{IV}}\email{iauthor@gmail.com}
%%=============================================================%%

\author*[1]{\fnm{Nicola Rares} \sur{Franco}}\email{nicolarares.franco@polimi.it}

\affil*[1]{\orgdiv{MOX, Department of Mathematics}, \orgname{Politecnico di Milano}, \orgaddress{\street{P.zza Leonardo da Vinci 32}, \city{Milan}, \postcode{20133}, \state{Italy}}}

\abstract{
We present a unified theoretical framework for parametric low-rank approximation, %where the ultimate goal is to develop 
\review{a research area devoted to the development of} efficient algorithms that act as adaptive alternatives of traditional methods such as Singular Value Decomposition (SVD), Proper Orthogonal Decomposition (POD), and Principal Component Analysis (PCA).
Applications include, e.g., the numerical treatment of parameter-dependent partial differential equations, where operators vary with parameters, and the statistical analysis of longitudinal data, where complex measurements, like audio signals and images, are collected over time. Recently, several adaptive algorithms have emerged, but a common mathematical foundation is still lacking, and existing solutions remain constrained to specific applications. As a result, key theoretical questions—--such as the existence and  regularity of optimal parametric low-rank approximants—--remain inadequately addressed. Our goal is to bridge this gap between theory and practice by establishing a rigorous framework for parametric low-rank approximation under minimal assumptions, specifically focusing on cases where parameterizations are either measurable or continuous. The analysis is carried out within the context of separable Hilbert spaces, ensuring applicability to both finite and infinite-dimensional settings. Finally, connections to recently emerging trends in the Deep Learning literature, relevant for engineering and data science, are also discussed.
}


%%\pacs[JEL Classification]{D8, H51}

%%\pacs[MSC Classification]{35A01, 65L10, 65L12, 65L20, 65L70}
\newtheorem{lemma}{Lemma}[section]
\newtheorem{corollary}{Corollary}[section]
\newcommand{\blabla}{\textit{Work in progress...}}
\newcommand{\gparam}{\xi}
\newcommand{\gspace}{\Xi}
\newcommand{\rank}{\textnormal{rank}}

\newcommand{\review}[1]{#1}
\newcommand{\latest}[1]{#1}

\newcommand{\betab}{\boldsymbol{\beta}}
\newcommand{\sqrtt}[1]{\left(#1\right)^{1/2}}
\newcommand{\omegafp}{(\mathbb{S},\mathcal{A},\mathbb{P})}
\newcommand{\bl}{\mathscr{BL}}
\newcommand{\expe}{\mathbb{E}}
\newcommand{\psd}{\mathsf{PSD}}
\newcommand{\scl}[2]{\langle#1,#2\rangle}
\newcommand{\wto}{\rightharpoonup}
\newcommand{\trace}{\mathscr{T}}
\newcommand{\hilbs}{\mathscr{H}}
\newcommand{\cmpts}{\mathscr{K}}
\newcommand{\schatten}{\mathscr{S}}
\newcommand{\opnorm}[1]{{\left\vert\kern-0.25ex\left\vert\kern-0.25ex\left\vert #1
    \right\vert\kern-0.25ex\right\vert\kern-0.25ex\right\vert}}
\newcommand{\opnormsmall}[1]{{\vert\kern-0.25ex\vert\kern-0.25ex\vert #1
    \vert\kern-0.25ex\vert\kern-0.25ex\vert}}
\newcommand{\hsnorm}[1]{\|#1\|_{\textnormal{HS}}}

\maketitle

\section{Introduction}

Techniques for low-rank approximation are ubiquitous in many areas of applied mathematics, from engineering, numerical analysis and linear algebra, where they are commonly employed to enhance the efficiency of numerical algorithms, to statistics and data science, where they offer reliable approaches for data compression and noise reduction. In this sense, low-rank approximation is a very broad term that can be related to a multitude of different problems. Here, we shall focus on two particular cases of primary importance for many applications.

The first one concerns the low-rank approximation of linear operators. For instance, when the ambient dimension is finite, given a matrix $\mathbf{A}\in\mathbb{R}^{N\times N}$, low-rank approximation techniques aim at finding a suitable surrogate, $\mathbf{A}_n\approx\mathbf{A}$ such that
$$\rank(\mathbf{A}_n)\ll\rank(\mathbf{A}).$$
A classical approach is to leverage the so-called Singular Value Decomposition (SVD). The latter is based on the idea that any matrix $\mathbf{A}\in\mathbb{R}^{N\times N}$ can be decomposed as
$$\mathbf{A}=\mathbf{V}\boldsymbol{\Sigma}\mathbf{U}^\top,$$
where $\mathbf{U},\mathbf{V}\in\mathbb{R}^{N\times r}$ are orthonormal matrices, whereas $\boldsymbol{\Sigma}=\text{diag}(\sigma_1,\dots,\sigma_r)\in\mathbb{R}^{r\times r}$ is diagonal, with entries sorted such that $\sigma_1\ge\dots\sigma_r\ge0$. Here, $r:=\rank(\mathbf{A})$. Then, for any $n < r$, a suitable low-rank approximant can be found by truncating the SVD as
$$\mathbf{A}_n:=\mathbf{V}_n\boldsymbol{\Sigma}_n\mathbf{U}_n^\top,$$
where, $\mathbf{U}_n$ and $\mathbf{V}_n$ are obtained by neglecting the last $r-n$ columns of $\mathbf{U}$ and $\mathbf{V}$, respectively. Similarly, $\boldsymbol{\Sigma}_n:=\text{diag}(\sigma_1,\dots,\sigma_n).$ This approximation can be shown to be optimal in both the spectral and the Frobenius norm \cite{eckart1936approximation}. %We also mention that, 
Under suitable assumptions, the same ideas can be extended to the \review{case of compact operators in infinite-dimensional Hilbert spaces}. %infinite-dimensional case, \review{leveraging For instance, one can leverage the same construction in order to find low-rank approximants of compact operators in arbitrary Hilbert spaces.

\review{The second application of interest, instead, concerns the problem of} 
%As we mentioned, another popular application of low-rank approximation techniques concerns 
data compression, or, equivalently, dimensionality reduction for random variables in high-dimensional settings. For instance, given an $N$-dimensional random vector $X$, one might be interested in finding a lower-dimensional representation of $X$, denoted as $\beta$, of dimension $n\ll N$. To this end, a classical approach consists in finding a suitable basis $\mathbf{V}\in\mathbb{R}^{N\times n}$, such that
$$X\approx \mathbf{V}\beta.$$
Here, the matrix $\mathbf{V}$ is deterministic, whereas the vector of coefficients, $\beta$, is random and it serves the purpose of modeling the stochasticity in $X$. In the literature, a popular algorithm for this task is the so-called Proper Orthogonal Decomposition (POD) \cite{quarteroni2015reduced}. Simply put, the latter looks for the matrix $\mathbf{V}$ that minimizes the mean squared projection error,
$$\mathbb{E}\|X-\mathbf{V}\mathbf{V}^\top X\|^2,$$
where $\mathbb{E}$ denotes the expectation operator. Then, the lower-dimensional representation is defined as $\beta:=\mathbf{V}^\top X$. The solution to such problem is known in closed form, and it ultimately involves finding a low-rank approximation of the (uncentered) covariance matrix $\mathbf{C}=(c_{i,j})_{i,j=1}^{N}$, 
$$c_{i,j}:=\mathbb{E}[X_iX_j],$$
where $X_i$ represents the $i$th component of the random vector $X$. We mention that, in the context of statistical applications, POD is usually replaced with a similar algorithm called Principal Component Analysis (PCA) \cite{jolliffe1990principal}. In this case, the random vector $X=(X_1,\dots,X_N)$ is first standardized as $\tilde{X}=(\tilde{X}_1,\dots,\tilde{X}_N)$, where either\footnote{We adopt the common convention according to which $\mathbb{E}^{1/2}[\;\cdot\;]:=\sqrt{\mathbb{E}[\;\cdot\;].}$}
$$\tilde{X}_i:=\frac{X_i-\mathbb{E}[X_i]}{\mathbb{E}^{1/2}|X_i-\mathbb{E}[X_i]|^2},$$
or $\tilde{X}_i:= X_i-\mathbb{E}[X_i]$, depending on the problem at hand. Then, a POD is performed over the standardized vector $\tilde{X}$.
%
%
%In general,
As in the case of linear operators, both POD and PCA have a natural generalization to the case of Hilbert-valued random variables, where $X$ can be infinite-dimensional: see, e.g., \cite{singler2014new}, Functional Principal Component Analysis \cite{ramsay2005principal, horvath2012inference} and the Kosambi-Karhunen-Loève expansion \cite{kosambi2016statistics, karhunen1947under}.

\subsection{\bf Parameter dependent problems}
While classical low-rank approximation techniques are very well understood, both theoretically and practically, things become more subtle when we move to parametrized scenarios, which, however, are of remarkable importance. For instance, in parameter-dependent partial differential equations (PDEs), it is common to encounter linear operators $A = A_\gparam$ that depend on a specific parameter $\gparam$. We can think of, e.g., the volatility coefficient in a diffusion process, or the viscosity coefficient in a fluid-flow simulation. Likewise, in the discrete setting, one frequently encounters parameter-dependent matrices $\mathbf{A} = \mathbf{A}_\gparam$. Of note, this also includes the case of nonautonomous dynamical systems, where the time variable $t$ acts as the parameter for the evolution operator, meaning that $\gparam = t$: see, e.g., \cite{musharbash2015error}.

Parametric dependency also arises quite naturally when considering high-dimensional random variables. For instance, aside from the case of longitudinal data (time varying), a typical application is that of \textit{conditional observations}. Assume, e.g., that $X$ is a random vector in $\mathbb{R}^N$, whereas $Y$ is a random variable taking values in a suitable set $\gspace$. Then, if $X$ and $Y$ are observed jointly, one can be interested in the conditional distribution of $X$ given $Y$, which naturally brings one to consider the family of random vectors $\{X_\gparam\}_{\gparam\in\gspace}$ defined as
$$X_\gparam:=X\mid Y = \gparam.$$
%Equivalently, given any measurable set $R\subset\mathbb{R}^{N}$, let $\phi_{R}:\gspace\to\mathbb{R}^{N}$ be such that $\phi_{R}(Y)=\mathbb{E}[\mathbf{1}_{R}(X)|Y].$ Then, $X_\gparam$ is a random variable whose probability distribution satisfies
%$$\mathbb{P}(X_\gparam\in R)=\frac{\phi_R(\gparam)}{\phi_{\mathbb{R}^N}(\gparam)}.$$
%In any case, we see that the values attained by $Y$ can now be interpreted as parameters that parametrize the random vector $X_\gparam$.
Equivalently, %we see that this change in perspective brings 
one can think of $Y=\gparam$ as a contextual variable parametrizing %the random vector
$X_\gparam$: see, e.g., \cite{gupta2018parameterized}.
\\\\
In both scenarios, practical applications, such as uncertainty quantification, optimal control and precision medicine, which are characterized by the necessity of exploring the parameter space, may demand for \textit{parametric} low-rank approximation. %In principle, this could be addressed with the aforementioned techniques (SVD and POD), up to repeating the computation for every $\gparam\in\gspace$. However, while this approach would yield optimal low-rank approximations, the associated computational cost can easily become unbearable, pushing domain practitioners towards alternative strategies. Mathematically speaking, this boils down to finding suitable surrogates that can replicate the performances of the maps
In principle, this issue could be tackled using the previously mentioned techniques (SVD and POD) by repeating the computation for every $\gparam\in\gspace$. While this method would provide optimal low-rank approximations, the associated computational cost can quickly become prohibitive, thus bringing domain practitioners to look for alternative approaches. Mathematically speaking, this challenge translates into the search for effective surrogates that can emulate the performance of the maps
\begin{equation}
    \label{eq:ideal-optimum}
    \gspace\ni\gparam\mapsto\mathbf{A}_{n}(\gparam)\in\mathbb{R}^{N\times n}\quad\text{and}\quad\gspace\ni\gparam\mapsto\mathbf{V}(\gparam)\in\mathbb{R}^{N\times n},
\end{equation}
corresponding to parameter-wise SVD and POD, respectively\footnote{Note: in practice, the rigorous definition of these maps may require user-defined preferences. In fact, depending on multiplicities, SVD and POD truncations may fail to be unique.} (hereon also referred to as \emph{parametric SVD} and \emph{parametric POD});  two maps that are optimal in theory but prohibitively expensive to compute in practice.
%
%one might be interested in finding a parameter dependent low-rank approximation of the object, be it a linear operator or a random variable. It is clear that, for every fixed $\gparam\in\gspace$, one could leverage the classical theory to construct a suitable approximation. However, this approach can be highly inefficient, as one needs to repeat the computation for every parameter instance $\gparam.$ 
%
%Then, a better alternative is to directly look for an algorithm that maps each $\gparam$ onto a suitable low-rank approximation of the object under study. We mention that, 
As of today, several approaches have emerged in this direction. For instance, in the case of linear operators, adaptive versions of SVD and related methods have been proposed in \cite{bunse1991numerical, berman2024colora}. Similarly, there has also been an increasing interest in deriving conditional/parametric versions of PCA and POD, as well as dynamical ones: see, e.g., \cite{cardot2007conditional, gupta2018parameterized, amsallem2008interpolation,  amsallem2012nonlinear, boncoraglio2021active, boncoraglio2021model, franco2024deep, koch2007dynamical, sapsis2009dynamically,  peherstorfer2020model}.  However, 
driven by specific applications, these approaches fail to recognize the existence of a common mathematical structure, and, most importantly, do not consider the fact that adaptive approaches are ultimately approximations of the optimal algorithms in Eq. \eqref{eq:ideal-optimum}, whose approximability is dictated by their parametric regularity. To the best of our knowledge, these issues were only partially addressed in the specific case of analytic parametrizations featuring a single scalar parameter $\gparam$ (allowed to be either real or complex): see, e.g., \cite{bunse1991numerical}. %in the case in which $\gparam$ is one-dimensional, be it real or complex, has been addressed throughfully, especially in the case of analytic parametrizations.
In fact, this scenario is of fundamental importance in the so-called Perturbation Theory, of which a comprehensive overview is found in the celebrated book by T. Kato \cite{kato2013perturbation}. However, things quickly become more complicated when additional parameters are introduced, or when the regularity assumptions are weakened, cf. \cite{turner1969perturbation, kato2013perturbation}.

%certain theoretical aspects, such as the existence of optimal low-rank approximants and their parametric regularity, have only been partially investigated. That is, most of the theory available in the literature is either devoted to special cases or it focuses on other mathematical aspects. For instance, the case in which $\gparam$ is one-dimensional, be it real or complex, has been addressed throughfully, especially in the case of analytic parametrizations. In fact, this scenario is of fundamental importance in the so-called Perturbation Theory, of which a comprehensive overview is found in the celebrated book by T. Kato \cite{kato2013perturbation}. However, things quickly become more complicated when additional parameters are introduced, or when the regularity assumptions are weakened: see, e.g., \cite{kato2013perturbation}. To the best of our knowledge, this general setting has been barely investigated, at least not in a comprehensive way. For instance, the aforementioned contributions are highly heterogeneous as they range from finite-rank problems to infinite-dimensional ones, often relying on completely different techniques. Furthermore, all of them, except for \cite{bunse1991numerical} which relies on \cite{kato2013perturbation}, do not address the problem of parametric regularity. 

In this sense, our work aims to take a step further by establishing a common theoretical foundation that can withstand minimal assumptions. Specifically, we will examine fundamental regularity properties ---parametric measurability and continuity--- of the ideal algorithms described in Eq. \eqref{eq:ideal-optimum}, and explore their implications for practical applications involving parametric low-rank approximation.
This will concern both the case of linear operators and that of high-dimensional random variables. To ensure a broader applicability of our results, we frame our analysis within the context of separable Hilbert spaces, thereby addressing both infinite and finite-dimensional settings. In general, the ideas explored in this work will necessitate of basic results from Set-Valued Analysis, Functional Analysis, and Operator Theory, of which the reader can find a suitable reference in \cite{aubin2009set}, \cite{yosida2012functional} and \cite{conway2019course}, respectively.

\subsection{\bf Outline of the paper}
The paper is organized as follows. First, in Section~\ref{sec:preliminaries}, we \review{introduce some notation and recall some preliminary results on low-rank approximation and parametric minimization problems.}
%present some auxiliary results concerning the parametric regularity of minimum problems and minimal selections. 
%More precisely, given a parametrized objective function $J=J(x,c)$, we shall derive conditions under which the map $$f(x)=\inf_{c\in C}J(x,c),$$ is continuous, and conditions under which there exists either a measurable or continuous map $c_*=c_*(x)$ such that $$f(x)=J(x,c_*(x)).$$
%Then, in Section~\ref{sec:preliminaries}, we move to low-rank approximation, setting the proper notation and recalling basic results on SVD and POD.
Things are then put into action in Section~\ref{sec:results}, where we address the problem of parametric low-rank approximation under minimal regularity assumptions; in terms of novelty, this is our main contribution. Finally, in Section~\ref{sec:corollaries}, \review{we explore the practical implications of our findings, focusing on numerical algorithms that are based upon universal approximators—such as deep neural networks. In doing so, we also present numerical experiments that support and illustrate our conclusions.}
Concluding remarks are reported in Section~\ref{sec:conclusions}. For the sake of better readability, technical proofs and supplemental results are postponed to the Appendix.

\section{Preliminaries and notation}
\label{sec:preliminaries}

In this Section we provide a synthetic overview of the fundamental concepts and notions required to properly address the problem of low-rank approximation in Hilbert spaces. %, going from linear operators (Section~\ref{subsec:linear_operators}) to random variables (Section~\ref{subsec:random_variables}). 
Specifically, we take the chance to introduce some notation and present the general ideas behind two fundamental algorithms, SVD and POD, that are commonly employed in nonparametric settings. \review{Then, at the end of this Section, we recall some useful facts about parametrized minimization problems, which are propedeutical to our derivation.}

\paragraph{\review{Notation for compact operators}}

Given \review{two separable Hilbert spaces  $(H_1,\|\cdot\|_{H_1})$ and $(H_2,\|\cdot\|_{H_2})$, we denote by $\bl(H_1,H_2)$ the space of bounded linear operators from $H_1$ to $H_2$, equipped with the operator norm
$$\opnorm{T}:=\sup_{x\in B_{H_1}}\|T(x)\|_{H_2},$$
where $B_{H_1}=\{x\in H_1\;:\;\|x\|_{H_1}\le1\}$ is the unit ball in $H_1$. Given any $T\in \bl(H_1,H_2)$, we define the \textit{rank} of $T$, and we write $\rank(T)$, for the dimension of the image $T(H_1)\subseteq H_2$. We write  $\cmpts(H_1,H_2)$ to intend the subspace of $\bl(H_1,H_2)$ consisting of \textit{compact operators}, namely,
$$\cmpts(H_1,H_2):=\overline{\{T\in\bl(H_1,H_2)\;:\;\rank(T)<+\infty\}}^{\opnorm{\cdot}},$$
cf. \cite[Theorem VI.13]{reed1980methods}.
Following the characterization by  Allahverdiev, see \cite[Theorem 2.1]{gohberg1978introduction}, given any $A\in\cmpts(H_1,H_2)$, we define its $n$th singular value, and we write $\sigma_n(A)$, as
\begin{equation}
    \label{eq:svalues}
    \sigma_{n}(A):=\inf\left\{\opnorm{A-L}\;:\;L\in\cmpts(H_1,H_2),\;\rank(L)\le n-1\right\}.
\end{equation}
It is well-known that for all $A\in\cmpts(H_1,H_2)$ the sequence $\{\sigma_{n}(A)\}_{n=1}^{+\infty}$ is bounded, nonincreasing and vanishing to 0 for $n\to+\infty$. Furthermore, $\opnorm{A}=\sigma_1(A).$ Stronger properties involving the $p$-summability of the singular values, instead, give rise to \textit{Schatten class} operators.  Specifically, for all $1\le p<+\infty$, we set
$$\schatten_p(H_1,H_2):=\left\{A\in\cmpts(H_1,H_2)\;:\;\sum_{n=1}^{+\infty}\sigma_n(A)^{p}<+\infty\right\}\subset \cmpts(H_1,H_2),$$
which, we recall, are all separable Banach spaces under the \latest{Schatten} $p$-norm \cite{conway2019course}
$\|A\|_p:=\left(\sum_{n=1}^{+\infty}\sigma_n(A)^{p}\right)^{1/p}.$}
\review{Whenever $H_1=H_2=H$, we also introduce a special notation for the case of \emph{trace class} and \emph{Hilbert-Schmidt} operators,
$$\trace(H):=\schatten_1(H,H)\quad\text{and}\quad\hilbs(H):=\schatten_2(H,H),$$
respectively.}
We recall that, for all trace class operators $A\in\review{\trace(H)}$, one has $\sum_{i=1}^{+\infty}|\scl{Ae_i}{e_i}|<+\infty$ for \latest{any} orthonormal basis $\{e_i\}_i\subset H$.
Furthermore, \latest{the} \textit{trace} of the operator $A$, i.e.
$\mathsf{Tr}(A):=\sum_{i=1}^{+\infty}\scl{Ae_i}{e_i},$ solely depends on $A$ and it is actually independent of the basis $\{e_i\}_i$, cf. \cite{conway2019course}.

\review{$\hilbs(H)$ %, instead, 
is a Hilbert space under the inner product}
$$\scl{A}{B}_\text{HS}:=\mathsf{Tr}(A^*B)=\mathsf{Tr}(B^*A).$$
Here, $A^*$ denotes the \emph{adjoint} of $A$, that is, the unique operator satisfying $\scl{Ax}{y}=\scl{x}{A^*y}$ for all $x,y\in H$. It is well known that $A$ and $A^*$ share the same singular values \review{\cite[Chapter V.3]{kato2013perturbation}}; in particular, they belong to the same Schatten class . \review{The same holds even if $H_1\neq H_2.$}

%We also mention that, for $1<p<+\infty$, $\schatten_{p}(H)$ is the topological dual of $\schatten_{q}(H)$, where $1<q<+\infty$ is the unique value for which $$\frac{1}{p}+\frac{1}{q}=1.$$ Instead, for $p=1$, $\trace=\schatten_1(H)$ can be characterized as the topological dual of $(\cmpts,\opnorm{\cdot})$. In all such cases, the duality is realized through the trace operator \cite{conway2019course}, that is, via the dual product

%As a direct consequence, we have the following useful identities concerning the trace norm,$$\|A\|_1=\sup_{\substack{C\in \cmpts\\\opnorm{C}\le1}}\;\mathsf{Tr}(C^*A)=\sup_{\substack{C\in \cmpts\\\opnorm{C}\le1}}\;\mathsf{Tr}(CA),$$ and the Hilbert-Schmidt norm, $$ \|A\|_{\textnormal{HS}}^2=\mathsf{Tr}(A^*A)=\mathsf{Tr}(AA^*).$$

\paragraph{\review{Notation for Hilbert-valued random variables}}

Let $(H,\|\cdot\|)$ be a separable Hilbert space. Given a separable probability space $\omegafp$, where $\mathbb{S}$ is the sample space, $\mathcal{A}$ a given sigma field, and $\mathbb{P}$ a probability distribution, \review{we define} an $H$-valued random variable \review{to be} a Borel measurable map\footnote{Hereon, the existence of the probability space $(\mathbb{S},\mathcal{A},\mathbb{P})$ will be omitted: we will simply say that $X$ is an $H$-valued random variable, without specifying the underlying probability space (which, for simplicity, we always assume to be complete). Clearly, when considering $L^p_H$ spaces, the probability space is intended to be the same for all $H$-valued random variables under study.} $$X:\omegafp\to(H,\|\cdot\|).$$
For any exponent $p\in[1,+\infty)$, we denote by
$$L^p_H:=\left\{X\;H\textnormal{-valued r.v.}\;:\;\mathbb{E}\|X\|^p<+\infty\right\},$$
the \review{Bochner spaces of $p$-integrable $H$-valued random variables}, where $\mathbb{E}$ is the expectation operator, defined with respect to $\mathbb{P}.$ \review{Up to quotients,} the latter are all separable Banach spaces under the norm
$$\|X\|_{L^p_H}:=\left(\mathbb{E}\|X\|^p\right)^{1/p},$$
%as soon as %the probability space $(\mathbb{S},\mathcal{A},\mathbb{P})$ is itself separable. %, that is, if the metric $d_{\mathscr{A}}(A,B):=\mathbb{P}(A\setminus B)+\mathbb{P}(B\setminus A)$ makes $(\mathscr{A}, d_{\mathscr{A}})$ a separable topological space.
\review{Given $X\in L^1_H$, we denote by $\mathbb{E}[X]$ the unique element in $H$ satisfying
$\scl{\mathbb{E}[X]}{z}=\mathbb{E}[\scl{X}{z}]$ for all $z\in H,$
that is, the Bochner integral of $X$ \cite[Chapter V.5]{yosida2012functional}.}

\subsection{Fundamentals of low-rank approximation}
\review{We now recall some fundamental results in the theory of low-rank approximation, starting with the case of linear operators and then moving to Hilbert-valued random variables.}

\subsubsection{\review{Linear operators: the Singular Value Decomposition (SVD)}}

The SVD, also known as \emph{polar form}, or \emph{canonical form} of compact operators, is a powerful tool that allows one to represent compact operators as an infinite series (resp., sum, if the operator has finite rank) of rank-1 operators: see, e.g., \cite[Theorem VI.17]{reed1980methods}. Precisely, if \review{$A:H_1\to H_2$} is a compact operator, then one has the representation formula
\begin{equation}
    \label{eq:svd}
    A=\sum_{i=1}^{r} s_i\scl{\cdot}{u_i}_{H_1}v_i,
\end{equation}
where $r=\rank(A)\in\mathbb{N}\cup\{+\infty\}$, 
for suitable $s_1\ge s_2\ge\dots\ge0$, and orthonormal sets \review{$\{u_i\}_i\subset H_1$ and $\{v_i\}_i\subset H_2.$} Notably, $s_i=\sigma_i(A)$ are the singular values of $A$. The vectors $u_i$ and $v_i$ are often called the left and right singular vectors of $A$, respectively. If $A$ is self-adjoint, then $u_i=v_i$; in particular, singular values and singular vectors coincide with the notion of eigenvalues and eigenvectors \review{\cite[Chapter V.3]{kato2013perturbation}.}

From the perspective of low-rank approximation, SVD is extremely useful as it allows to easily identify optimal low-rank approximants. In fact, one can prove that for every $n\le r$, the best $n$-rank approximation of $A$ is given by truncating the series in Eq. \eqref{eq:svd} at $i=n.$ For later reference, we formalize this fact in the Lemma right below, which is ultimately a more abstract version of the well-known Eckart-Young Lemma \cite{eckart1936approximation}. \review{Notably, while the representation formula \eqref{eq:svd} is commonly found in standard textbooks—see, for example, \cite[Chapter V.3]{kato2013perturbation} or \cite[Exercise VI.47]{reed1980methods}—discussions on the optimality and uniqueness of the singular value decomposition (SVD) are typically restricted to the finite-dimensional setting, as in \cite{eckart1936approximation}, and are rarely addressed in the broader context of separable Hilbert spaces. To bridge this gap, we provide the interested reader with a complementary proof of the Lemma presented below: see Appendix~\ref{appendix:classical-proofs}.}\\

\begin{lemma}
    \label{lemma:svd}
    Let \review{$(H_1,\|\cdot\|_{H_1})$ and $(H_2,\|\cdot\|_{H_2})$ be two separable Hilbert spaces. Let $A\in\cmpts(H_1, H_2)$}. There exist two orthonormal sequences, \review{$\{u_i\}_{i=1}^{+\infty}\subset H_1$ and $\{v_i\}_{i=1}^{+\infty}\subset H_2$}, respectively, such that
    $A=\sum_{i=1}^{+\infty}\sigma_i(A)\scl{\cdot}{u_i}_{H_1}v_i.$
    Additionally, for every $n\in\mathbb{N}_+$, the operator $A_n:=\sum_{i=1}^{n}\sigma_i(A)\scl{\cdot}{u_i}_{H_1}v_i$ satisfies
    \begin{equation}
    \label{eq:spectral-minimizer}
        \opnorm{A-A_n}=\inf_{\substack{L\in\cmpts(H)\\\rank(L)\le n}}\;\opnorm{A-L}=\sigma_{n+1}(A),
    \end{equation}
    and
    \begin{equation}
    \label{eq:hs-minimizer}
    \hsnorm{A-A_n}^2=\inf_{\substack{L\in\cmpts(H)\\\rank(L)\le n}}\;\hsnorm{A-L}^2=\sum_{i=n+1}^{+\infty}\sigma_i(A)^2.\end{equation}
    Furthermore, if \review{$A\in\hilbs(H_1, H_2)$} and $\sigma_{n+1}(A)<\sigma_n(A)$, then $A_n$ is the unique minimizer of \eqref{eq:hs-minimizer}.
\end{lemma}
\begin{proof}
    See Appendix~\ref{appendix:classical-proofs}.
\end{proof}

\subsubsection{\review{Random variables: the Proper Orthogonal Decomposition (POD)}}

As we mentioned in the introduction, POD is popular technique for reducing the complexity of high-dimensional random variables. From an abstract point of view, the construction underpinning the POD is based on a truncated series expansion. The idea, in fact, is that all square-integrable Hilbert-valued random variable $X$ admit a series representation of the form
\begin{equation}
\label{eq:pod expansion}
X=\sum_{i=1}^{+\infty}\sqrt{\lambda_i}\eta_i v_i,
\end{equation}
where $\lambda_1\ge\lambda_2\ge\dots\ge0$ is a nonincreasing sequence of positive scalar numbers, $\{v_i\}_{i=1}^{+\infty}$ is an orthonormal basis of the Hilbert state space $H$, whereas $\{\eta_i\}_{i=1}^{+\infty}$ is an $\mathbb{E}$-orthonormal sequence of scalar valued random variables, meaning that $\mathbb{E}[\eta_i\eta_j]=\delta_{i,j}$. We mention that, if $H=L^{2}(\Omega)$ for a suitable spatial domain $\Omega\subseteq\mathbb{R}^{d}$, then $X$ can be considered a \textit{random field}, and the series expansion is often referred to as "Kosambi-Karhunen-Loève expansion", \review{in relation to the works by Karhunen K. \cite{karhunen1947lineare}, Loève M. \cite{loeve1978probability} and Kosambi D.\cite{kosambi2016statistics}, which originally addressed this problem in the framework of function spaces. Concerning early works with a more abstract treatment, instead, we mention \cite{payen1967fonctions}, which directly addressed the case of Hilbert spaces.}

It is worth highlighing the fact that, although one has $X=\sum_{i=1}^{+\infty}\scl{X}{w_i}w_i$ %$\mathbb{P}$-
almost surely for any orthonormal basis $\{w_i\}_i\subset H$,  optimal representations are only obtained for special choices of the basis vectors (which, ultimately, depend on $X$). In fact, in most cases the random variables $\omega_i=\scl{X}{w_i}$ would yield $\mathbb{E}[\omega_i\omega_j]\neq0$. In particular, if $\mathbb{E}[X]=0$, this would result in a statistical correlation between the coefficients in the series expansion.

As a matter of fact, it is this avoidance of redundances that makes the \review{\eqref{eq:pod expansion} remarkably useful for low-rank approximation \cite{payen1967fonctions}}. In fact, given a reduced dimension $n$, one can prove the following optimality of the truncated series expansion (and, thus, of the POD),
 $$\mathbb{E}\left\|X-\sum_{i=1}^{n}\sqrt{\lambda_i}\eta_i v_i\right\|^2=\inf_{Z\in Q_n}\mathbb{E}\|X-Z\|^2,$$
        where $Q_n=\{Z\in L^{2}_H\;:\;\exists V\subseteq H,\;\dim(V)\le n,\;Z\in V\;%\mathbb{P}\text{-almost surely}
        \textnormal{almost surely}\}$\review{: see, e.g., \cite{payen1967fonctions}. For later reference, we formalize these facts in the Lemma below. In doing so, we also include related results that will come in handy later on. The interested reader can find a detailed proof in the Appendix~\ref{appendix:classical-proofs}, or directly refer to \cite{payen1967fonctions}.}\\

\begin{lemma}
    \label{lemma:pod}
    Let $(H,\|\cdot\|)$ be a separable Hilbert space. Given a square-integrable $H$-valued random variable $X$, i.e. $\mathbb{E}^{1/2}\|X\|^2<+\infty$, let $B:H\to H$ be the linear operator
    $$B:u\mapsto \mathbb{E}\left[\scl{u}{X}X\right],$$
    where the integral is understood in the Bochner sense. Then:
    
    \begin{itemize}
        \item[i)] $B\in\trace$ in a symmetric positive semidefinite trace class operator;
        \item[ii)] there exists a sequence of (scalar) random variables $\{\eta_{i}\}_{i=1}^{+\infty}$ with $\expe[\eta_i\eta_j]=\delta_{i,j}$ such that
        $$X=\sum_{i=1}^{+\infty}\sqrt{\lambda_i}\eta_iv_i$$
        %$\mathbb{P}$-
        almost surely, where $\lambda_1\ge\lambda_2\ge\dots\ge0$ and $v_i\in H$ are the eigenvalues and eigenvectors of $B$, respectively;
        \item[iii)] for every orthogonal projection $P:H\to H$ one has
        $\mathbb{E}\|X-PX\|^2=\mathbb{E}\|X\|^2-\sum_{i=1}^{+\infty}\lambda_i\|Pv_i\|^2;$
        \item[iv)] for every $n\in\mathbb{N}_+$, $X_n:=\sum_{i=1}^{n}\sqrt{\lambda_i}\eta_iv_i$ satisfies
        $\mathbb{E}\|X-X_n\|^2=\inf_{Z\in Q_n}\mathbb{E}\|X-Z\|^2,$
        where $Q_n=\{Z\in L^{2}_H\;:\;\exists V\subseteq H,\;\dim(V)\le n,\;Z\in V\;%\mathbb{P}\text{-almost surely}
        \textnormal{almost surely}\}.$
    \end{itemize}
\end{lemma}
\begin{proof}
    See Appendix~\ref{appendix:classical-proofs}.
\end{proof}

\subsection{\review{Measurability and continuity of the \emph{argmin} map}}
\label{sec:min-argmin}

\review{We conclude this Section by stating some useful results related to the solution of parametrized minimization problems, with a major emphasis on the regularity of the minimizers with respect to the underlying problem parameters.
Some of these results will provide the foundation for our forthcoming analysis in Section~\ref{sec:results}, while others are intended to deepen our intuition on the subject. }

%In this Section we derive some auxiliary results, concerning the regularity of parametrized minimization problems, which are commonly used in mathematical economics and optimal control, see, e.g., \cite{stokey1989recursive, ok2007real}. Since our presentation will be slightly more abstract than the one provided in classical textbooks, we shall provide the reader with suitable proofs tailored for our purposes.

%More precisely, 
\review{Let $\gparam\in\gspace$ and $c\in C$ denote the problem parameters and the optimization variable, respectively, of a suitable minimization problem with cost functional $J=J(\gparam, c).$}
%Given a parametrized objective function $J=J(\gparam,c)$, 
We shall \review{state} %derive 
conditions under which: \emph{(i)} \review{for each $\gparam\in\gspace$ there exists at least a minimizer $c\in C$; \emph{(ii)} there exists a map $c_*=c_*(\gparam)$, which is either  measurable or continuous in $\gparam$,} called \textit{minimal selection}, such that $$f(\gparam)=J(\gparam,c_*(\gparam))\quad \forall\gparam\in\gspace.$$

\review{We anticipate that, in our framework, $\gparam$ will be a suitable parameter parametrizing either a compact operator or a random variable. Conversely, $c$ will be a vector collecting all of the constituents necessary for the construction of a low-rank approximant (basis vectors, coefficients), whereas $J$ will be some cost functional measuring the accuracy of the approximation, in connection with the optimality results in Lemmae~\ref{lemma:svd}--\ref{lemma:pod}}.
\\\\
\review{In general, whether a parametric minimization problem admits a minimizer that depends continuously on the problem parameters, is a challenging question. This is because, even under fairly strong assumptions ---such as the compactness of $\gspace$ and $C$, and the smoothness of the objective functional $J$--- a continuously varying minimizer may fail to exist. In this respect, a crucial role is played by the \emph{multiplicity} of the solutions. We illustrate this fact with a simple example, depicted in Figure~\ref{fig:example-minimization}.}

\review{There, we consider a parametric minimization problem with}
$$J(\gparam,c)=(\gparam c)^3-\gparam c,$$
where $c\in C:=\in[-1,1]$ and $\gparam\in\gspace:=[1,2]$. \review{As made evident by the picture, any minimal selector $c_*:[1,2]\to[-1,1]$ would be discontinuous at $\gparam_0=2/\sqrt{3}$. In fact, it must be}
$$\review{c_*(\gparam)=}\begin{cases}
    \review{1/(\sqrt{3}\gparam)} & \review{1\le\gparam<\gparam_0,}\\
    \review{-1} & \review{\gparam_0<\gparam\le 2.}
\end{cases}$$
Since $J$ is analytic in both $\gparam$ and $c$, this shows that the smoothness of the objective function does not automatically translate into a smooth dependency of the minimizers with respect to the parameters. \review{Here, the issue}
\latest{lies in the fact, for $\gparam=\gparam_0$, the solution to the minimization problem is not unique.}
%is posed by uniqueness of the solution to the minimization problem, which holds over $\gspace\setminus\{\gparam_0\}$ but not across the whole parameter space $\gspace$.}

\begin{figure}
    \centering
    \includegraphics[width=0.55\linewidth]{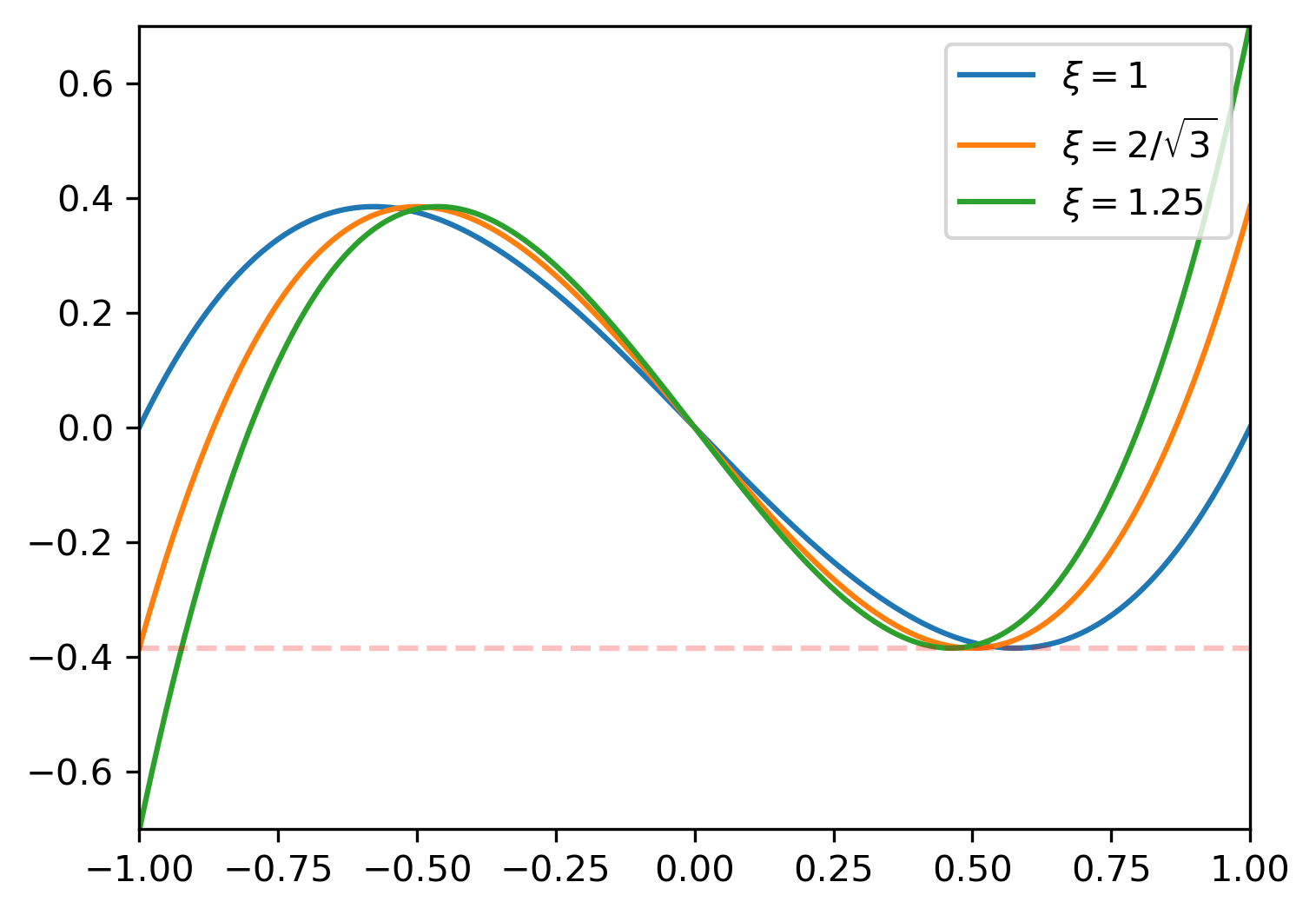}
    \caption{\small Profile of the objective function $J(\gparam,c)=(\gparam c)^3-\gparam c$ for three values of $\gparam$. In orange, the branching case $\gparam=2/\sqrt{3}$, where two minima appear.}
    \label{fig:example-minimization}
\end{figure}

%
%As we shall prove in Section~\ref{subsec:continuity}, this issue is intrinsic of minimization problems with nonunique solutions. 

%As typical of the classical literature on minimization problems, the main ingredients of our analysis will involve lower-semicontinuity and compactness. %We recall, in fact, the following basic result, which is contained in most books of mathematical analysis (although often stated in a weaker form: see, e.g., \cite[Theorem 5.4.3]{lafferriere2022introduction}).

%\begin{lemma}   \label{lemma:min}Let $(C,d_{C})$ be a compact metric space. Let $g:C\to\mathbb{R}$ be lower semi-continuous. Then, there exists some $c_*\in C$ such that$$g(c_*)=\inf_{c\in C}g(c).$$\end{lemma}

%\begin{proof}
 %   Let $\{c_n\}_n\subset C$ be such that $\lim_n g(c_n)=\inf_{c\in C}g(c).$ Up to passing to a subsequence, \review{due to} compactness of $C$, there exists some $c_*\in C$ such that $c_n\to c_*$. Since $g$ is lower semi-continuous, $g(c_*)\le\liminf_n g(c_n)=\inf_{c\in C}g(c),$ and the conclusion follows.
%\end{proof}

\review{Intuitively, one expects a measurable minimal selection to exist in most cases, and a continuous one to be possible when the minimization problem admits a unique solution for all parametric instances $\gparam\in\gspace.$ These considerations are made rigorous in the Theorem below.}\\

\begin{theorem}[\bf Argmin regularity]
    \label{theorem:argmin}
    Let $(\gspace,d_{\gspace})$ and $(C,d_{C})$ be two Polish spaces, with $C$ compact. Let $J:\gspace\times C\to \mathbb{R}$ be lower semi-continuous. Assume that the map $\gparam\mapsto J(\gparam,c)$ is continuous for every $c\in C$.
    Then, there exists a Borel measurable map $c_*:\gspace\to C$ such that
    $$J(\gparam,c_*(\gparam))=\min_{c\in C}J(\gparam,c)\quad\quad\forall \gparam\in \gspace.$$  
    Additionally, if every $\gparam\in\gspace$ admits a unique $c'\in C$ minimizing $J(\gparam,\cdot)$, then $c_*$ is also continuous.
\end{theorem}

\begin{proof} 
See Appendix~\ref{appendix:classical-proofs}
\end{proof}

\review{We remark that the results in Theorem~\ref{theorem:argmin} are far from being novel. The one concerning the measurability of the \emph{argmin} map, for instance, is closely related to other results easily found in the literature, see, e.g., \cite[Theorem 18.19]{aliprantis2006infinite} and \cite[Theorem 14.37]{rockafellar1998variational}, whereas the one on continuity is a special case of Berge's Theorem \cite[Theorem 17.31]{aliprantis2006infinite}. Nonetheless, we provide the interested reader with suitable proofs that can be found in the Appendix~\ref{appendix:classical-proofs}. This is because, in the measurable case, our statement does not readily follow from \cite[Theorem 18.19]{aliprantis2006infinite} ---which assumes continuity in $c$, rather than lower-semicontinuity--- nor from \cite[Theorem 18.19]{aliprantis2006infinite} ---which assumes $C$ to be a subset of $\mathbb{R}^n$. Here, in fact, in light of our forthcoming analysis in Section~\ref{sec:results}, we are specifically interested in results that can be applied in the general context of Hilbert spaces and for which lower-semicontinuity suffices (essentially, because we will need to leverage the compactness provided by the weak topology, under which the norm is only lower-semicontinuous). Concerning the continuity result, instead, the proof reported in the Appendix is intended for the curious reader: under our assumptions, in fact, the statement can be proven using rather elementary arguments, without having to resort to the more advanced machinery of Set-Valued analysis underpinning Berge's Theorem.}

%\begin{corollary}
%    \label{corollary:convex}
%    Let $(X,\|\cdot\|_X)$ and $(\gspace,\|\cdot\|_\gspace)$ be Hilbert spaces. Let $J:X\times \gspace\to \mathbb{R}$ be continuous and convex. Let $C\subset \gspace$ be weakly compact. Then, the map $f:X\to\mathbb{R}$ defined as
%    $$f(x):=\inf_{c\in C}J(x,c)$$
%    is continuous.
%\end{corollary}
%\begin{proof}
%    Let $\tau_X$ and $\tau_\gspace$ be the weak topologies over $X$ and $\gspace$, respectively.
%    It is well known that, \review{due to} convexity and strong continuity, $J$ is lower semi-continuous with respect to the product topology $\tau_{X}\otimes\tau_\gspace.$ Now, let $d_{C}$ be a suitable metric over $C$ compatible with the restriction of $\tau_\gspace$ to $C$, so that $(C,d_{C})$ is compact. Define, $\tilde{J}:X\times C\to\mathbb{R}$ as $\tilde{J}(x,c)=J(x,c)$. Clearly, $\tilde{J}$ is lower semi-continuous with respect to the product metric induced by $\|\cdot\|_{X}$ and $d_{C}$. In fact, if $x_n\to x\in X$ and $y_n\wto y\in C$, then we also have $x_n\wto x$, and we can leverage the (weak) lower semi-continuity of $J$ with respect to $\tau_X\otimes \tau_\gspace.$ We may then apply Theorem~\ref{theorem:continuity} to $X,C,\tilde{J}$ and obtain the desired conclusion.
%\end{proof}

\section{Parameter dependent low-rank approximation}
\label{sec:results}

We are now ready to address the case of parameter dependent low-rank approximation. \review{We begin by outlining some established results before presenting our main contribution. The latter will be split into three parts: first, a result of general interest, and then a separate focus on parametric SVD and parametric POD, respectively.}

%We mention that, throughout the whole Section we will need to use basic facts of Functional Analysis concerning weak topologies. To this end, we recall that, given a Hilbert space $(H,\|\cdot\|)$ the weak topology $\tau_{H}$ is the coarsest topology that makes the linear functionals $l_{x}:y\mapsto\scl{x}{y}$ continuous for all $x\in H$. To distinguish between convergence in the strong (i.e., norm) topology and the weak topology, we will use the notation $v_k\to v$ and $v_k\wto v$, respectively. We also recall that: all convex sets are strongly closed if and only if they are weakly closed; all weakly closed sets that are norm bounded are weakly compact (and metrizable, if $H$ is separable); all weakly convergent sequences are norm bounded; compact operators map weakly convergent sequence into strongly convergent ones; the weak and strong topology induce the same Borel sigma-field. %, cf. Lemma~\ref{lemma:borel} in Appendix~\ref{appendix:auxiliary}.

\review{In order to motivate our analysis, we find it instructive to start with a simple example that illustrates how, although singular values may vary smoothly with the underlying operator, singular vectors can exhibit discontinuities}. \review{Specifically, let us consider the following} parameter dependent $2\times 2$ matrix,
$$\mathbf{A}_{\gparam}=\left[
\begin{array}{cc}
    \gparam & 0 \\
     0 & 1-\gparam \\
\end{array}\right],$$
where $0\le\gparam\le1.$ \review{In this case,} singular values and singular vectors coincide with eigenvalues and eigenvectors, respectively. Since $\sigma_{1}(\mathbf{A}_\gparam)$ is the largest singular value, we have 
$$
\sigma_{1}(\mathbf{A}_{\gparam})=\begin{cases}
    1-\gparam & \text{if}\;0\le\gparam\le0.5\\
    \gparam & \text{if}\;0.5<\gparam\le1.
\end{cases}
$$
Similarly, $\sigma_{2}(\mathbf{A}_{\gparam})=\gparam$ if $0\le\gparam\le0.5$ and $\sigma_{2}(\mathbf{A}_{\gparam})=1-\gparam$ otherwise. Both $\sigma_1$ and $\sigma_2$ are continuous. However, the corresponding eigenvectors are
$$v_1^\gparam=\begin{cases}
    [0,1]^\top & \text{if}\;0\le\gparam\le0.5\\
    [1,0]^\top & \text{if}\;0.5<\gparam\le1.\end{cases}\quad\text{and}\quad
v_2^\gparam=\begin{cases}
    [1,0]^\top & \text{if}\;0\le\gparam\le0.5\\
    [0,1]^\top & \text{if}\;0.5<\gparam\le1,\end{cases}$$
both of which depend discontinuously on $\gparam.$ Clearly, the issue is caused by the branching point $\gparam=0.5$, where $\sigma_{1}(\mathbf{A}_\gparam)=\sigma_{2}(\mathbf{A}_\gparam)$. From an intuitive point of view, this is not really surprising. This whole phenomenon, in fact, is strictly related to the one discussed in Section~\ref{sec:min-argmin}: indeed, we can think of singular values and singular vectors as \review{solutions to} suitable minimization problems. Then, it becomes evident that changes in multiplicites of the eigenvalues can produce discontinuities in the dependency of singular vectors. %Still, by leveraging the results in Section~\ref{sec:min-argmin}, we may at least trade continuity with measurability. We detail our reasoning below, starting with the case of singular values and them moving to singular vectors.\\

\review{Having this in mind, the following results, which are part of the literature on Perturbation Theory, become rather intuitive.}

\begin{theorem}
    \label{theorem:singularvalues}
    Let \review{$(H_1,\|\cdot\|_{H_1})$ and $(H_2,\|\cdot\|_{H_1})$ be two separable Hilbert spaces}, and fix any $n\in\mathbb{N}_+$. 
    The map $\sigma_{n}:\cmpts(H_1,H_2)\to\mathbb{R}$ is 1-Lipschitz continuous. That is, for all $A,B\in\cmpts(H_1,H_2)$ one has
    $$|\sigma_n(A)-\sigma_n(B)|\le\opnorm{A-B}.$$
\end{theorem}
\begin{proof}
    \review{This is a classical result, see, e.g., \cite{kato2013perturbation}. Nonetheless, it can be proven effortlessly using the characterization in Eq. \eqref{eq:svalues}. Indeed, let} $n\in\mathbb{N}_+$ and $A,B\in\cmpts(H_1,H_2)$. Fix any $\varepsilon>0$. By definition, there exists some $L_\varepsilon\in\cmpts(H_1,H_2)$ with $\rank(L)\le n-1$ such that $\opnorm{B-L_\varepsilon}-\varepsilon<\sigma_n(B)$. We have
    %\begin{multline*}
    \begin{equation*}
    \sigma_n(A)-\sigma_n(B)<\sigma_n(A)-\opnorm{B-L_\varepsilon}+\varepsilon\le %\\\le
    \opnorm{A-L_\varepsilon}-\opnorm{B-L_\varepsilon}+\varepsilon\le\opnorm{A-B}+\varepsilon.
    \end{equation*}
    %\end{multline*}
    Since $\varepsilon>0$ was arbitrary, we deduce $\sigma_n(A)-\sigma_n(B)\le\opnorm{A-B}$. As the situation is symmetric in $A$ and $B$, the conclusion follows.
\end{proof}

\begin{theorem}
    \label{theorem:rieszprojection}
    \review{Let $(H_1,\|\cdot\|_{H_1})$ and $(H_2,\|\cdot\|_{H_1})$ be two separable Hilbert spaces. Let $n\in\mathbb{N}_+$ and $\varepsilon>0.$ Consider the open set
    $$\mathscr{O}_{n,\varepsilon}:=\{A\in\cmpts(H_1,H_2)\;:\;\sigma_{n}(A)-\sigma_{n+1}\latest{(A)}>\varepsilon\}\subset \cmpts(H_1,H_2).$$
    Let $\mathscr{P}_n:=\{P\in\cmpts(H_2)\;:\;\rank(P) = n,\;P^2=P,\;P=P^*\}$ be the set of orthonormal projections from $H_2\to H_2$ of $n$-rank. \latest{Then, there}
    %There 
    exists a continuous map $\cmpts(H_1,H_2)\ni A\to P\in\mathscr{P}_n$ that, to each operator $A$, assigns the orthonormal projection onto the subspace spanned by the first $n$ left (right) singular vectors of $A$. The same holds if the Hilbert-Schmidt norm is considered and $\cmpts(H_1, H_2)$ is replaced with $\hilbs(H_1, H_2).$}
\end{theorem}

\begin{proof}
    \review{This is essentially \cite[Theorem 3.16]{kato2013perturbation} applied to the self-adjoint operator $T:=AA^*$ (or $A^*A$, depending on whether one considers left or right singular vector). Since the map $A\to T$ is continuous, the conclusion follows. Concerning the Hilbert-Schmidt norm, notice that $\hilbs(H_1,H_2)$ embedds continuously in $\cmpts(H_1,H_2)$. Additionally, for every $P,Q\in\mathscr{P}_n$ one has $\hsnorm{P-Q}\le n\opnorm{P-Q}$, meaning that $(\mathscr{P}_n,\opnorm{\cdot})$ embeds continuously in $(\mathscr{P}_n,\hsnorm{\cdot})$. The statement is then trivial.}
\end{proof}

\review{Notice that, in order to talk about continuity, Theorem~\ref{theorem:rieszprojection} requires a \emph{gap condition}. This is very intuitive if we consider the interplay between the uniqueness result in Lemma~\ref{lemma:svd} and the continuity statement in Theorem~\ref{theorem:argmin}. In fact, it is an interesting exercise to prove Theorem~\ref{theorem:rieszprojection} using only Theorem~\ref{theorem:argmin} and not by relying on the more advanced theory of Functional Calculus (in contrast to the approach adopted in \cite{kato2013perturbation}). Additionally, we note that Theorem~\ref{theorem:rieszprojection} addresses the continuous dependence of the projection operator itself, rather than providing a continuity statement about the singular vectors. This is because, while the subspace spanned by the singular vectors is univocally defined, the corresponding basis is not. In what follows, we will deal with this issue using the theory of \emph{fiber bundles} \cite{husemoller1966fibre}.}

\review{Beside that, our purpose for the remainder of this Section is to provide complementary results to Theorem~\ref{theorem:rieszprojection} in order to show that: \emph{(i)} when the gap condition fails, parametric measurability still holds; \emph{(ii)} under suitable assumptions, one can recover an explicit parametrization of singular vectors, thus unpacking the underlying projection operator.}

\subsection{General \review{result}}
\review{We begin with a general result on the measurability of singular vectors, which is one of the main novelties in this work. We derive the latter as an application of Theorem~\ref{theorem:argmin}.}\\

\begin{theorem}
\label{theorem:measurab}
Let \review{$(H_1,\|\cdot\|_{H_1})$ and $(H_2,\|\cdot\|_{H_2})$} be two separable Hilbert spaces. Fix any $n\in\mathbb{N}_+$. There exists a Borel measurable map from $\cmpts(H_1,H_2)\to H_1^{n}\times H_2^n$, mapping
$$A\mapsto (u_1^A,\dots,u_n^A,v_1^A,\dots,v_n^A),$$
such that, for every $A\in\cmpts(H_1,H_2)$,
\begin{itemize}
    \setlength\itemsep{1em} 
    \item [i)] $Au_i^{A}=\sigma_i(A)v_i^{A}$ and $A^*v_i^{A}=\sigma_i(A)u_i^{A}$ for all $i=1,\dots,n$;
     \item [ii)] $A_n:=\sum_{i=1}^{n}\sigma_i(A)\langle\cdot,u_i^A\rangle_{H_1} v_i^A,$
minimizes the $n$-rank truncation error, i.e.,
$$\opnorm{A - A_n} = \inf\left\{\opnorm{A-L}\;:\;L\in\cmpts(H_1,H_2),\;\rank(L)\le n\right\};$$
    \item [iii)] for all $i,j=1,\dots,n$ with $i,j\le\rank(A)$ one has $\scl{u_i^A}{u_j^A}_{H_1}=\scl{v_i^A}{v_j^A}_{H_2}=\delta_{i,j}$;
    \item [iv)] the optimality in (ii) holds also with respect to the Hilbert-Schmidt norm;
    \item [v)] if $A=A^*$ then $u_i^A=v_i^A$ for all $i=1,\dots,n$ with $\sigma_i(A)>0.$
\end{itemize}

\end{theorem}
\begin{proof}
    We subdivide the proof into multiple steps. In what follows, let \review{$B_{H_j}:=\{w\in H_j\;:\;\|w\|_{H_j}\le1\}$ be the closed unit ball in $H_j$, with $j=1,2$. Define $C:=(B_{H_1})^n\times (B_{H_2})^n$ and equip the latter with the product topology $\tau_C:=\tau_{H_1}^n\otimes\tau_{H_2}^n$, with $\tau_{H_j}$ being the weak topology over $B_{H_j}$. In what follows, to distinguish between weak and strong convergence we shall use $\wto$ and $\to$, respectively.}
       
    \;\\{\bf Step 1.}
    \textit{The functional $J_0:(\review{\cmpts(H_1,H_2)},\opnorm{\cdot})\times (C,\tau_C)\to\mathbb{R}$ acting as
    $$J_0(A,u_1,\dots u_n,v_1,\dots, v_n)=\opnorm{A - \sum_{i=1}^{n}\sigma_i(A)\langle\cdot,u_i\rangle v_i},$$
    is: i) marginally continuous in $A$; ii) lower-semicontinuous. 
    }

    \begin{proof}
    The first statement is trivial thanks to Theorem~\ref{theorem:singularvalues}. As for the lower semi-continuity, instead, let $A_k\to A$, $u_i^{(k)}\wto u_i$ and $v_i^{(k)}\wto v_i$. Then, for every \review{$(x,y)\in H_1\times H_2$}, $\|x\|_{\review{H_1}}=1$, we have
    $$\sum_{i=1}^{n}\sigma_i(A_k)\langle x,u_i^{(k)}\rangle_{\review{H_1}} \langle v_i^{(k)},y\rangle_{\review{H_2}}\to \sum_{i=1}^{n}\sigma_i(A)\langle x,u_i\rangle_{\review{H_1}} \langle v_i,y\rangle_{\review{H_2}}.$$
    Thus, $A_kx-\sum_{i=1}^{n}\sigma_i(A_k)\langle x,u_i^{(k)}\rangle_{\review{H_1}} v_i^{(k)}\wto Ax-\sum_{i=1}^{n}\sigma_i(A)\langle x,u_i\rangle_{\review{H_1}}  v_i,$
    since $A_kx\to Ax$ strongly.
    Consequently, if \review{we let $c=(u_1,\dots,u_n,v_1,\dots,v_n)$ and $c_k:=(u_1^{(k)},\dots,u_n^{(k)},v_1^{(k)},\dots,v_n^{(k)}),$}
      %\begin{multline*}
      \begin{equation*}
      \left\|Ax-\sum_{i=1}^{n}\sigma_i(A)\langle x,u_i\rangle_{\review{H_1}}  v_i\right\|\le%\\\le
      \liminf_{k\to+\infty}\left\|A_kx-\sum_{i=1}^{n}\sigma_i(A_k)\langle x,u_i^{(k)}\rangle_{\review{H_1}} v_i^{(k)}\right\|\le\liminf_{k\to+\infty}\review{J_0}(A_k,c_k).\end{equation*}
      %\end{multline*}
    Passing at the supremum over $\|x\|=1$ yields $J_0(A,c)\le\liminf_{k\to+\infty}J_0(A_k,c_k).$\qedhere
    \end{proof}

    \;\\{\bf Step 2.}
    \textit{The functional $J_{1,1}:(\cmpts\review{(H_1, H_2)},\opnorm{\cdot})\times (C,\tau_C)\to\mathbb{R}$ acting as
    $$J_{1,1}(A,u_1,\dots u_n,v_1,\dots, v_n)=\sum_{i=1}^{n}\|Au_i-\sigma_i(A)v_i\|_{\review{H_2}},$$
    is: i) marginally continuous in $A$; ii) lower-semicontinuous. 
    }

    \begin{proof}
    As previously, the first statement comes directly from Theorem~\ref{theorem:singularvalues}. As for the second one, instead, let $A_k\to A$, $u_i^{(k)}\wto u_i$ and $v_i^{(k)}\wto v_i$. We have
    $$\|A_ku_i^{(k)}-Au_i\|_{\review{H_2}}\le \opnorm{A_k-A}\cdot\|u_i^{(k)}\|_{\review{H_1}}+\|Au_i^{(k)}-Au_i\|_{\review{H_2}}\to 0,$$
    since $u_i^{(k)}$ is uniformly norm bounded and $A$ turns weak convergence into strong convergence. Thus, meaning that $A_ku_i^{(k)}\to Au$ strongly. In particular $A_ku_i^{(k)}+\sigma_i(A_k)v_i^{(k)}\wto A u_i+\sigma_i(A)v_i$. Since $\|\cdot\|_{\review{H_2}}$ is lower-semicontinuous in the weak topology, the conclusion follows.
    \end{proof}

    \;\\{\bf Step 3.}
    \textit{The functional $J_{1,2}:(\cmpts\review{(H_1, H_2)},\opnorm{\cdot})\times (C,\tau_C)\to\mathbb{R}$ acting as
    $$J_{1,2}(A,u_1,\dots u_n,v_1,\dots, v_n)=\sum_{i=1}^{n}\left|\sigma_i(A)\sigma_j(A)\delta_{i,j}-\scl{Au_i}{Au_j}_{\review{H_2}}\right|^2,$$
    is continuous.
    }

    \begin{proof}
        Let $A_k\to A$ and $u_i^{(k)}\wto u_i$. As before, we have $A_ku_i^{(k)}\to Au$ strongly \review{due to} compactness of $A$. Since $\sigma_i$ is continuous in $A$, the conclusion follows.
    \end{proof}

    \;\\{\bf Step 4.}
    \textit{The functional $J_{2}:(\cmpts\review{(H_1, H_2)},\opnorm{\cdot})\times (C,\tau_C)\to\mathbb{R}$ acting as
    \begin{align*}
        J_{2}(A,u_1,\dots u_n,v_1,\dots, v_n)=J_{1,1}(A^*,v_1,\dots,v_n,u_1,\dots,u_n)+J_{1,2}(A^*,v_1,\dots,v_n,u_1,\dots,u_n),
    \end{align*}
    is: i) marginally continuous in $A$; ii) lower-semicontinuous. 
    }

    \begin{proof}
        We recall that $A_k\to A$ implies $A_k^*\to A^*.$ Then, the statement follows from Steps 2-3.
    \end{proof}

    \;\\{\bf Step 5.} \textit{Statements (i)-(iii) in  Theorem~\ref{theorem:measurab} hold true.}

    \begin{proof}
    Consider the functional $J(\cmpts\review{(H_1, H_2)},\opnorm{\cdot})\times (C,\tau_C)\to\mathbb{R}$ defined as
    $$J(A,c)=J_0(A,c)+J_{1,1}(A,c)+J_{1,2}(A,c)+J_2(A,c),$$
    which, by the previous steps, is lower-semicontinuous in $(A,c)$ and continuous in $A$ for every $c\in C.$ We  notice that, for every $A\in\cmpts\review{(H_1, H_2)}$,$$J_0(A,c)\ge\sigma_{n+1}(A),\quad J_{1,1}\ge0,\quad J_{1,2}\ge0,\quad J_{2}\ge0.$$ 
    In addition, for every $A\in\cmpts\review{(H_1, H_2)}$, one can leverage the classical SVD to construct a list $c_*=(u_1,\dots,u_n,v_1,\dots,v_n)$ that minimizes all four functionals simultaneously: $J_0(A,c_*)=\sigma_{n+1}(A)$ \review{due to} optimality of the truncated SVD; $J_{1,1}(A,c_*)=0$ \review{due to} definition of singular vectors; $J_{1,2}(A,c_*)=0$ thanks to the orthogonality constraints, $$\scl{Au_i}{Au_j}_{\review{H_2}}=\sigma_i(A)\sigma_j(A)\scl{v_i}{v_j}_{\review{H_2}}=\sigma_i(A)\sigma_j(A)\delta_{i,j};$$
    and similarly for $J_2$, since $\scl{A^*v_i}{A^*v_j}_{\review{H_1}}=\sigma_i(A)\sigma_j(A)\scl{u_i}{u_j}_{\review{H_1}}$. Clearly, $c_*$ is not unique as, for instance, $J(A,c_*)=J(A,-c_*)$. Nonetheless, this shows that
    $$\inf_{c\in C}J(A,c)=\sigma_{n+1}(A)=\inf\left\{\opnorm{A-L}\;:\;L\in\hilbs,\;\rank(L)\le n\right\}.$$
    Then, leveraging the continuity properties of $J$ and the compactness of $(C,\tau_C)$, \review{we may now exploit Theorem~\ref{theorem:argmin} to construct a Borel measurable selector} $c_*:\cmpts\review{(H_1, H_2)}\to C$ such that $J(A,c_*(A))=\inf_{c\in C}J(A,c)=\sigma_{n+1}(A)$ for all $A\in\cmpts\review{(H_1, H_2)}$. \review{Notice that we are able to apply Theorem~\ref{theorem:argmin} since the separability of $H_1$ and $H_2$ ensures that $\tau_C$ is both metrizable and separable.}
    For better readability, write $(u_1^A,\dots,u_n^A,v_1^A,\dots,v_n^A):=c_*(A).$
    Notice that $J(A,c_*(A))=\sigma_{n+1}(A)$ implies
    \begin{itemize}
        \setlength{\itemsep}{1em}
        \item $J_{1,1}(A,c_*(A))=0$, thus $Au_i^A=\sigma_i(A)v_i^A$. This is part of (i) in the Theorem;
        \item $J_0(A,c_*(A))=\sigma_{n+1}(A)$, which is (ii) in the Theorem;
         \item $J_{1,2}(A,c_*(A))=0$ and, by the previous observation, 
         $$\sigma_{i}(A)\sigma_{j}(A)\delta_{i,j}=\scl{Au_i^A}{Au_j^A}_{\review{H_2}}=\sigma_{i}(A)\sigma_{j}(A)\scl{v_i}{v_j}_{\review{H_2}}.$$
         In particular, if $i,j\le\rank(A)$ then $\sigma_i(A),\sigma_j(A)>0$ and the above yields $\scl{v_i}{v_j}_{\review{H_2}}=\delta_{i,j}$. This is part of (iii) in the Theorem;
         \item $J_{2}(A,c_*(A))=0$ ensuring the equivalent conditions on the adjoint and thus proving (i) and (iii).\hfill\qedhere
    \end{itemize}
    \end{proof}

    \;\\{\bf Step 6.} \textit{Statements (iv)-(v) in  Theorem~\ref{theorem:measurab} hold true.}

    \begin{proof}Thanks to (i)-(iii), the vectors $u_i^A$ and $v_i^A$ are guaranteed to be left and right singular vectors of $A$, respectively, whenever $\sigma_i(A)>0$. In particular, the operator $A_n$ is a truncated SVD of $A$, and (iv) automatically follows from Lemma~\ref{lemma:svd}. Finally, to prove (v) notice that if $A=A^*$ then
    $$A^2u_i^A=A(\sigma_i(A)v_i^A)=\sigma_i(A)Av_i^A=\sigma_i(A)^2u_i^A,$$
    meaning that $u_i^A$ is an eigenvector of $A^2$. However, since $A$ is symmetric, $A$ and $A^2$ are known to share the same eigenvectors. Then, it must be $\sigma_i^Av_i^A=Au_i^A=\sigma_i(A)u_i^A.$
    Dividing by $\sigma_i(A)>0$ yields the desired conclusion. \let\qed\relax\hfill\qedsymbol\;\;\qedsymbol\end{proof}\let\qed\relax
\end{proof} 

%\begin{remark} As a direct consequence of Theorems~\ref{theorem:singularvalues} and~\ref{theorem:measurab}, it is evident that there exists a measurable map $\cmpts\to\cmpts$ that, for a fixed dimension $n$, maps each $A$ onto a corresponding optimal low-rank approximant $A_n$ (according to the operator norm). Similarly, there also exists a measurable map $\hilbs\to\hilbs$ that assigns each $A$ onto an optimal low-rank approximant $A_n$ defined according to the Hilbert-Schmidt norm. Interestingly, in the latter case, by leveraging Theorem~\ref{theorem:continuity} combined with the uniqueness result in Lemma~\ref{lemma:svd}, one can actually prove that the restriction of such map to the open subset $\mathscr{O}_n:=\{A\in\hilbs\mid\sigma_{n+1}(A)<\sigma_{n}(A)\}$ is continuous. Here, however, we shall discuss this fact directly in Section~\ref{subsec:parametric svd} when considering the parametric scenario. \end{remark}

\subsection{\bf Parametric SVD}
\label{subsec:parametric svd}

We \review{are now ready to state and prove a measurability and a continuity result for parametric SVD, respectively.}\\ %measurability result for parametric SVD by combining Theorems~\ref{theorem:singularvalues}-\ref{theorem:measurab}.\\

\begin{theorem}
    \label{theorem:svd}
    Let $(\gspace,\mathscr{M})$ be a measurable space\review{. Let $(H_1,\|\cdot\|_{H_1})$ and $(H_2,\|\cdot\|_{H_2})$ be two separable Hilbert spaces}. Let $\gspace\ni\gparam\to A_\gparam\in\review{\cmpts(H_1,H_2)}$ be a measurable map. Fix any $n\in\mathbb{N}_+$. Assume that $\rank(A_\gparam)\ge n$ for all $\gparam\in\gspace$. Then, there exist measurable maps
    $$s_i:\gspace\to[0,\infty),\;\;u_i:\gspace\to \review{H_1},\;\;v_i:\gspace\to \review{H_2},\;\;i=1,\dots,n,$$
    such that $s_i(\gparam)=\sigma_i(A_\gparam)$, $\langle u_i(\gparam),u_j(\gparam)\rangle\review{_{H_1}} = \langle v_i(\gparam),v_j(\gparam)\rangle\review{_{H_2}} = \delta_{i,j}$ and
    $$\opnorm{A_\gparam - \sum_{i=1}^{n}s_i(\gparam)\langle\cdot,u_i(\gparam)\rangle_{\review{H_1}} v_i(\gparam)} = \inf_{\substack{L\in\cmpts(H_1,H_2)\\\rank(L)\le n}}\opnorm{A_\gparam-L}=\sigma_{n+1}(A_\gparam),$$
    for all $\gparam\in \gspace.$ 
    In particular, there exists a measurable map $\gspace\ni\gparam\mapsto A_\gparam^n\in\cmpts(H_1,H_2)$ mapping each $\gparam$ onto an optimal $n$-rank approximation of $A_\gparam.$
    The same results hold if\review{, in all of the above,} $\review{\cmpts(H_1,H_2)}$ is replaced with $\review{\hilbs(H_1,H_2)}$ and the operator norm is substituted by the Hilbert-Schmidt norm.
\end{theorem}
\begin{proof} We limit the proof to $(\cmpts\review{(H_1,H_2)},\opnorm{\cdot})$ as the same arguments can be readily applied to the Hilbert-Schmidt case.
    Let $c_*:\cmpts(H_1,H_2)\to H_1^n\times H_2^n$ be the measurable map in Theorem~\ref{theorem:measurab}. For \review{$i=1,\dots,n$, let $p_i:H_1^{n}\to H_1$ and $q_i:H_2^n\to H_2$ be the projections} onto the $i$th component. Set
    $$s_i(\gparam):=\sigma_i(A_\gparam),\;\;u_i(\gparam):=p_i(c_*(A_\gparam)),\;\;v_i(\gparam):=q_{i}(c_*(A_\gparam)).$$
    Then, the conclusion follows by composition thanks to Theorem~\ref{theorem:singularvalues} and Theorem~\ref{theorem:measurab}. %Finally, notice that $\|u_i(\gparam)\|=\|v_i(\gparam)\|=1$ \review{due to} orthonormality. Thus, it is clear that $u_i,v_i\in L^{\infty}(\gspace;H)$.
\end{proof}

\review{As we anticipated, deriving a suitable continuity result entails an additional difficulty given by the fact that, while the truncated SVD is unique under the \emph{gap assumption}, its representation in terms of basis vectors is not. Still, if the parameter space is topologically contractible to a point, ---e.g., $\gspace=[0,1]^p$--- then we can still recover an explicit representation: see Theorem~\ref{theorem:continuous-svd} right below.}\\

\begin{theorem}
    \label{theorem:continuous-svd}
    Let $(\gspace,d_\gspace)$ be a compact metric space. \review{Let $(H_1,\|\cdot\|_{H_1})$ and $(H_2,\|\cdot\|_{H_2})$ be two separable Hilbert spaces. Let $\gspace\ni\gparam\to A_\gparam\in\cmpts(H_1,H_2)$ be continuous.} Fix any $n\in\mathbb{N}_+$ and assume that $\sigma_{n+1}(A_\gparam) < \sigma_n(A_\gparam)$ for all $\gparam\in\gspace.$ %Then, there 
    Then, there exists a continuous map
    $\gspace\ni\gparam\mapsto A_{\gparam}^n\in\cmpts(H_1,H_2)$ 
    such that, for every $\gparam\in\gspace$, the operator $A_{\gparam}^n$ is \review{an optimal $n$-rank approximation of $A_\gparam$ \review{in both the operator and the Hilbert-Schmidt norms}.
    Additionally, if $(\gspace, d_{\gspace})$ is contractible to a point, then, for $i=1,\dots,n$, there exist continuous maps
    $$\tilde{u}_i:\gspace\to H_1\quad\text{and}\quad v_i:\gspace\to H_2,$$
    such that, for every $\gparam\in\gspace$, one has  $\langle v_i(\gparam),v_j(\gparam)\rangle_{H_2}=\delta_{i,j}$ and \begin{equation}
    \label{eq:lowrankexpansion}A_{\gparam}^{n}=\sum_{i=1}^{n}\langle\cdot,\tilde{u}_i(\gparam)\rangle_{H_1}v_i(\gparam).\end{equation} 
    The same results hold if, in all of the above, $\review{\cmpts(H_1,H_2)}$ is replaced with $\review{\hilbs(H_1,H_2)}$ and the operator norm is substituted by the Hilbert-Schmidt norm.}
\end{theorem}

\begin{proof}
\review{Without loss of generality, we shall limit the proof the case of $(\cmpts(H_1,H_2),\opnorm{\cdot})$: the Hilbert-Schmidt case will follow similarly. Leveraging compactness and continuity, let $\varepsilon:=\min_{\gparam\in\gspace}|\sigma_{n+1}(A_\gparam)-\sigma_n(A_\gparam)|>0.$ For each $\gparam\in\gspace$, let us denote by $P_\gparam$ be the orthonormal projection onto the subspace spanned by the first $n$ left singular vectors of $A_\gparam$. Then, in virtue of Theorem~\ref{theorem:rieszprojection}, the Riesz projection $P_\gparam$ depends continuously on $\gparam$ under the operator norm. Define $A_\gparam^n:=P_\gparam A_\gparam.$ It is straightforward to see that $A_\gparam^n$ is the truncated SVD of $A_\gparam$. Furthermore, it is clear that $A_\gparam^n$ depends continuously on $\gparam$. The first part of the Theorem is thus proven.}

\review{For the second part, instead, consider the triplet $(E,p,B)$ given by
$$E:=\bigcup_{\gparam\in\gspace}\{\gparam\}\times P_\gparam(H_2),\quad\quad B:=\gspace,\quad\quad p:E\to B,$$
where $p:\{\gparam\}\times P_\gparam(H)\mapsto \gparam.$ We claim that $(E,p,B)$ is a \emph{vector bundle} according to the usual definition: see, e.g., Definition 1.1 in Chapter 3 of \cite{husemoller1966fibre}. Clearly, for every $\gparam\in B$, the fiber $p^{-1}(\gparam)$ is equipped with an $n$-dimensional vector structure, since $p^{-1}(\gparam)\cong P_\gparam(H_2).$ Thus, we only need to prove that for every $\gparam\in B$ there exists an open neighborhood $U\subseteq B$ and a map $\phi_{U}:U\times\mathbb{R}^n\to p^{-1}(U)$ such that, for every $\gparam'\in U$, the restriction to $\{\gparam'\}\times\mathbb{R}^n\to p^{-1}(\gparam')$ is an isomorphism of vector spaces.}

\review{To this end, fix any $\gparam\in B.$
Consider the map $f:\gparam'\mapsto \sigma_n(P_{\gparam'}P_\gparam)$, going from $B$ to $\mathbb{R}$. Since $P_{\gparam'}$ depends continuously of $\gparam'$ and $\sigma_n$ is continuous as well, we conclude that $f$ is also continuous. Given that
$$f(\gparam)=\sigma_n(P_\gparam P_\gparam)=\sigma_n(P_\gparam^2)=\sigma_n(P_\gparam)=1>0,$$
there must exist some $\delta>0$ such that $f(\gparam')>1/2$ for all $\gparam'\in B$ with $d_{\gspace}(\gparam,\gparam')<\delta.$ Define $$U:=\{\gparam'\in B\;:\;d_{\gspace}(\gparam,\gparam')<\delta\}.$$
Then, for all $\gparam'\in U$ we have $\sigma_n(P_{\gparam'}P_{\gparam})>0$, and thus, $\rank(P_{\gparam'}P_{\gparam})=n.$ This allows us to define the map $\phi_{U}:U\times\mathbb{R}^n\to p^{-1}(U)$ as
$$\phi_U:(\gparam',c)\mapsto\left(\gparam',\sum_{i=1}^{n}c_i P_{\gparam'}(v_i)\right)\in \{\gparam'\}\times P_{\gparam'}(H_2)$$
where $v_1,\dots, v_n$ is any basis of $P_\gparam(H_2),$ temporarily fixed for the sole purpose of defining $\phi_U$. Basically, the idea is to exploit the fact that, if we are close enough to $\gparam$, namely $\gparam'\in U$, then one can project a vector onto $P_{\gparam'}(H_2)$ in two steps: first projecting from $H_2$ to $P_{\gparam}(H_2)$, and then to $P_{\gparam'}(H_2)$. Indeed, since $\rank(P_{\gparam'}P_{\gparam})=n$, the set $\{P_{\gparam'}(v_i)\}_{i=1}^{n}$ is guaranteed to be a basis of $P_{\gparam'}(H_2)$, and no information is lost in the process. In particular, the restriction of $\phi_U$ to $\{\gparam'\}\times\mathbb{R}^{n}\to p^{-1}(\gparam')$ is an isomorphism. Since $\gparam'\in U$ was arbitrary (and so was $\gparam\in B$), this shows that $(E,p,B)$ is indeed a vector bundle.}

\review{As a next step, we notice that the base space of such vector bundle, $B=\gspace$, is contractible by hypothesis. Then, $(E,p,B)$ is \emph{trivial}: see, e.g., \cite[Corollary 4.8]{husemoller1966fibre}. In particular, there exist continuous functions $\tilde{v}_i:B\to H_2$ such that, for every $\gparam\in B$, the set $\{\tilde{v}_i(\gparam)\}_{i=1}^{n}$ is a basis for the fiber $p^{-1}(\gparam)$.} 

\review{Since $p^{-1}(\gparam)\cong P_{\gparam}(H_2),$ we have nearly obtained our desired parametrization, except for the orthonormality constraints. To address this fact, we define the maps $v_i:\gspace\to H_2$ by applying, pointwise in the parameter space, the Gram-Schmidt orthonormalization to the parametrized basis $\tilde{v}_i$. Namely, if $\varrho:H_2\setminus\{0\}\to H_2$ is the normalization map, $\varrho(x):=x/\|x\|_{H_2}$, then
\begin{itemize}
    \item[---] $v_1(\gparam):=\varrho(\tilde{v}_1(\gparam))$
    \item[---] $v_{i+1}(\gparam):=\varrho(\tilde{v}_{i+1}(\gparam)-\sum_{j=1}^{i}\scl{\tilde{v}_{i+1}(\gparam)}{v_j(\gparam)}_{H_2}\tilde{v}_{i+1}(\gparam))$ for $i=1,\dots,n-1.$
\end{itemize}
Notice that, since $\{\tilde{v}_i(\gparam)\}_i$ is always a basis of rank $n$, this procedure preserves the continuity with respect to $\gparam$. Finally, to conclude, we set
$$\tilde{u}_i(\gparam):= A^*_\gparam v_i(\gparam).$$
Then, for every $x\in H_1$, one has
$$
\sum_{i=1}^{n}\scl{x}{\tilde{u}_i(\gparam)}_{H_1}v_i(\gparam)=\sum_{i=1}^{n}\scl{x}{A^*_\gparam v_i(\gparam)}_{H_1}v_i(\gparam))=
\sum_{i=1}^{n}\scl{A_\gparam x}{v_i(\gparam)}_{H_2}v_i(\gparam))=P_\gparam(A_\gparam x)=A_\gparam^n x,$$
meaning that Eq. \eqref{eq:lowrankexpansion} holds, as claimed.}
\end{proof}

\subsection{\bf Parametric POD}

We now switch from linear operators to high-dimensional random variables, \review{addressing the case of} %with the purpose of deriving basic regularity results for 
parametric POD. \review{As before, we provide both a measurability and a continuity result.}
%In this case, we shall present two results. The first one discusses the measurability of parametric POD, and it is obtained under very mild assumptions. The second one, instead, is a continuity result, derived under a suitable no-branching condition. The two are discussed in Sections~\ref{subsec:pod-measurable} and~\ref{subsec:pod-continuous}, respectively. %Before coming to the actual Theorems, however, we permit a useful auxiliary Lemma. Simply put, the latter states that the (uncentered) covariance operator depends continuously on the underlying random variable. For the better readability, the proof is postponed to the Appendix, Section Appendix~\ref{appendix:auxiliary}.

\begin{theorem}
    \label{theorem:pod}
    Let $(\gspace,\mathscr{M})$ be a measurable space and let $(H,\|\cdot\|)$ be a separable Hilbert space. Let $\{X_\gparam\}_{\gparam\in \gspace}\subseteq L^{2}_{H}$ be a family of square-integrable Hilbert-valued random variables.
    %$$\mathbb{E}\|X_\gparam\|^2<+\infty\quad\forall\gparam\in \gspace.$$
    Assume that the map $\gparam\to X_\gparam$ is measurable from $\gspace\to L^2_H$. %to 
    %$$L^2_H:=\left\{Z\; \textnormal{r.v.}\; H\textnormal{-valued}:\;\mathbb{E}\|Z\|^2<+\infty\right\},$$ 
    %the latter being equipped with the Borel sigma-algebra generated by the norm $\|Z\|_{L^2_H}=\mathbb{E}^{1/2}\|Z\|^2.$
    Fix any $n\in\mathbb{N}_+.$ Then, there exist measurable maps $v_i:\gspace\to H$ for $i=1,\dots,n,$
    such that $\langle v_i(\gparam),v_j(\gparam)\rangle=\delta_{i,j}$ and 
    $$\mathbb{E}\left\|X_\gparam - \sum_{i=1}^{n}\langle X_\gparam, v_i(\gparam)\rangle v_i(\gparam)\right\|^2=\inf_{Z\in Q_n}\;\mathbb{E}\|X_\gparam-Z\|^2,$$
    for all $\gparam\in \gspace,$ where
    $Q_n = \left\{Z\in L^2_H\;:\;\exists V\subseteq H,\;\dim(V)\le n,\;Z\in V \text{almost surely}\right\}.$
    Furthermore, if $\gspace$ is a metric space and $\mathbb{E}\|X_\gparam-X_{\gparam'}\|^2\to 0$ for $\gparam\to\gparam'$, then the maps $\gparam\to\mathbb{E}|\langle X_\gparam, v_i(\gparam)\rangle|^2$ are continuous.
\end{theorem}

\begin{proof}
Fix any $\gparam\in \gspace$. Let $B_\gparam:H\to H$,
$$B_\gparam(u):=\mathbb{E}\left[\langle u,X_\gparam\rangle X_\gparam\right],$$
be the ---uncentered--- covariance operator of $X_\gparam$. Let $\lambda_i^\gparam$ and $v_i^\gparam$ be the eigenvalues and eigenvectors of $B_\gparam$, sorted such that $\lambda_{i+1}^\gparam\le\lambda_{i}.$ Set $\eta_i:=\scl{X_\gparam}{v_i^\gparam}/\sqrt{\lambda_i^\gparam}$. Then, as we discussed in Lemma~\ref{lemma:pod}, it is well-known that the random variable 
$$Z_\gparam^*=\sum_{i=1}^{n}\eta_{\gparam,i}\sqrt{\lambda_i^\gparam}v_i^\gparam= \sum_{i=1}^{n}\langle X_\gparam, v_i^\gparam\rangle v_i^\gparam,$$
minimizes $\mathbb{E}\|X_\gparam-Z\|^2$ within $Q_n$.
%where $\eta_{\gparam,i}$ are scalar valued random variables satisfying $\mathbb{E}[\eta_{\gparam,i}\eta_{\gparam,j}]=\delta_{i,j}$, whereas of the  
We now notice that, since $B_\gparam\in\review{\trace(H)}$, and $(\review{\trace(H)},\|\cdot\|_1)$ embedds continuously in $(\review{\cmpts(H)},\opnorm{\cdot})$, by Theorem~\ref{theorem:svd} there exist measurable maps $\tilde{s}_i$ and $v_i$ such that $$\tilde{s}_i(\gparam)=\sigma_i(B_\gparam)=\lambda_i^\gparam\quad\text{and}\quad v_i(\gparam)=v_i^\gparam,$$
The first statement in the Theorem follows. Finally, we notice that if $\mathbb{E}\|X_\gparam-X_{\gparam'}\|^2\to 0$ for $\gparam\to\gparam'$, then the map $\gparam\mapsto X_\gparam$ is actually $\gspace\to L^2_H$ continuous. Since 
$$\mathbb{E}\left|\langle X_\gparam, v_i(\gparam)\rangle\right|^2 = \mathbb{E}\left|\sqrt{\lambda_i^\gparam}\eta_{\gparam,i}\right|^2 = \lambda_i^{\gparam} = \sigma_i(B_\gparam),$$
and $B_\gparam$ depends continuously on $\gparam$ through $X_\gparam$ (cf. Lemma~\ref{lemma:covariance}) the conclusion follows.
\end{proof}

\begin{theorem}
    \label{theorem:continuous-pod}
    Let $(\gspace,d_\gspace)$ be a metric space and let $(H,\|\cdot\|)$ be a separable Hilbert space. Let $\{X_\gparam\}_{\gparam\in \gspace}$ be a family of square integrable $H$-valued random variables.
    Assume that the map $\gparam\to X_\gparam$ is continuous from $\gspace\to L^2_H$. Fix any $n\in\mathbb{N}_+$.
    For every $\gparam\in \gspace$ let
    $B_\gparam\in\hilbs$ be the (uncentered) covariance operator
    $$B_\gparam(u)=\mathbb{E}[\langle X_\gparam, u\rangle X_\gparam],$$
    and $\lambda_1^\gparam,\dots,\lambda_n^\gparam,\lambda_{n+1}^\gparam$ be its $n+1$ largest eigenvalues, $\lambda_i^\gparam=\sigma_i(B_\gparam)$.   
    If, for every $\gparam\in \gspace$, one has the strict inequality
    \begin{equation}
        \label{eq:gap-pod}
        \lambda_n^\gparam>\lambda_{n+1}^\gparam\ge0,
    \end{equation}
    then, there exists a continuous map $\review{\gspace\to\cmpts(H)}$, mapping
    $\gparam\mapsto P_{\gparam}$, \review{such that, for every $\gparam\in\gspace,$ the operator $P_\gparam$ is an orthonormal projection of rank $n$ and}
    \begin{equation}
    \label{eq:optimality}    
    \mathbb{E}\left\|X_\gparam -P_\gparam X_\gparam\right\|^2=\inf_{P\in \mathscr{P}_n}\;\mathbb{E}\|X_\gparam-PX_\gparam\|^2,\end{equation}
    \review{where $\mathscr{P}_n:=\{P\in\cmpts(H)\;:\;P^2=P,\;P=P^*,\;\rank(P)\le n\}.$
    Additionally, if $(\gparam,d_{\gspace})$ is contractible to a point, then, for $i=1,\dots,n$, there exist continuous maps $v_i:\gspace\to H$ such that for all $\gparam\in\gspace$ one has $\scl{v_i(\gparam)}{v_j(\gparam)}=\delta_{i,j}$ and $$P_\gparam=\sum_{i=1}^{n}\scl{\cdot}{ v_i(\gparam)}v_i(\gparam).$$}
\end{theorem}

\begin{proof}
    \review{As discussed in Lemma~\ref{lemma:pod} and in the proof of Theorem~\ref{theorem:pod}, for each $\gparam\in\gspace$, the best orthonormal projection of rank $n$ is the one over the subspace spanned by the first $n$ eigevectors associated to the largest $n$ eigenvalues of $B_\gparam$. Then, thanks to the gap assumption in Eq. \eqref{eq:gap-pod}, the first part of the Theorem follows readily from Theorem~\ref{theorem:rieszprojection}. Concerning the existence of a continuously parametrized orthonormal basis, instead, this is once again a consequence of classical results in the theory of fiber bundles. Indeed, the same ideas used in the proof of Theorem~\ref{theorem:continuous-svd} can be applied to this setting. The conclusion follows.}
\end{proof}

\subsubsection{Application to Kosambi-Karhunen-Loève expansions}

We conclude this subsection with an application of Theorems~\ref{theorem:pod}-\ref{theorem:continuous-pod} to the context of random fields, where each realization of the random variable $X$ is in fact a function, or a trajectory, defined over a suitable spatial domain. Specifically, we focus on the case in which $H=L^{2}(\Omega)$ for some measurable set $\Omega\subset\mathbb{R}^{d}.$ 
In this case, rather than introducing measurable maps $v_i:\gspace\to L^{2}(\Omega)$, we directly frame the result in terms of multi-variable functions $v_i:\gspace\times\Omega\to\mathbb{R}$, which better reflects the notation \review{that is} commonly adopted in the literature. \review{For the sake of simplicity, we restrict our attention to parameter spaces of the form $\gspace = [0,1]^p$ for some $p$. While this assumption is not entirely necessary, it serves to avoid an unduly verbose formulation of the result that follows.}\\ %Similarly, re-write the low-rank approximant using the classical formula derived from the Kosambi-Karhunen-Loève expansion.\\

\begin{corollary}[Parametric Kosambi-Karhunen-Loève expansion]
    \label{corollary:kkl}
    \review{Let $\gspace=[0,1]^p$} and let $\Omega\subseteq\mathbb{R}^{d}$ be Lebesgue measurable. Let $\{X_\gparam\}_{\gparam\in \gspace}$ be a family of stochastic processes defined over $\Omega$. 
    Assume that
    $\mathbb{E}\int_\Omega|X_\gparam(z)|^2dz<+\infty,$
    for all $\gparam\in \gspace.$ Additionally, assume that the map $(\gparam,z)\to X_\gparam(z)$ is measurable. Then, \review{for i=1,\dots, n, }there exist measurable maps
    %$$s_i:\gspace\to[0,\infty)\;\;\text{and}\;\;
    $v_i:\gspace\times\Omega\to \mathbb{R}$
    and a family of random variables $\{\tilde{\eta}_{\gparam,1},\dots,\tilde{\eta}_{\gparam,n}\}_{\gparam\in \gspace}$ with finite second moment, such that $\int_\Omega v_i(\gparam,z) v_j(\gparam,z)dz=\delta_{i,j}$ and
    \begin{equation}
    \label{eq:param-kkl}
    \mathbb{E}\int_\Omega|X_\gparam(z) - \sum_{i=1}^{n}\tilde{\eta}_{\gparam,i}v_i(\gparam,z)| ^2dz=\inf_{Z\in Q_n}\;\mathbb{E}\int_{\Omega}|X_\gparam(z)-Z(z)|^2dz,\end{equation}
    for all $\gparam\in \gspace,$ where
    $Q_n = \left\{Z=\sum_{i=1}^{n}a_if_i\;:\;\mathbb{E}|a_i|^2<+\infty,\;f_i\in L^{2}(\Omega)\right\}.$
    \review{Furthermore, if one has $\mathbb{E}\int_\Omega|X_\gparam(z)-X_{\gparam'}(z)|^2dz\to 0$ whenever $\gparam\to\gparam'$, then
    \begin{itemize}
        \setlength{\itemsep}{0.5em}
        \item [(i)] if the family of kernels $K_{\gparam}:\Omega^2\to\mathbb{R}$, defined a.e. in $\Omega^2$ as
        $K_{\gparam}(z,y):=\mathbb{E}\left[X_\gparam(z)X_\gparam(y)\right],$
        admits a strict and uniform separation between the first $n$ eigenvalues of the kernels and the rest of their spectrum (listed in decreasing order), then the $v_i$'s can be constructed such that       
        \begin{equation}
            \label{eq:l2-continuity}
            \lim_{\gparam'\to\gparam}\int_{\Omega}|v_i(\gparam,z)-v_i(\gparam',z)|^2dz=0\quad\forall i=1,\dots,n,
        \end{equation}
        for all $\gparam\in\gspace$.
        \item[(ii)] if (i) holds, $\Omega$ is compact and the family of kernels is uniformily Lipschitz continuous, i.e. there exists some $L>0$ such that $|K_\gparam(x,y)-K_\gparam(x',y')|\le L(|x-x'|+|(y-y')|$ for all $x,x',y,y'\in \Omega$ and all $\gparam\in\gspace,$ then the maps $v_i=v_i(\gparam,z)$ are marginally Lipschitz continuous in $z$ and jointly continuous in $(\gparam,z).$ Furthermore,
        \begin{equation}
        \label{eq:sup-continuity}\lim_{\gparam'\to\gparam}\;\sup_{z\in\Omega}|v_i(\gparam,z)-v_i(\gparam',z)|=0\quad\forall i=1,\dots,n,\end{equation}
        for all $\gparam\in\gspace.$
    \end{itemize}}
\end{corollary}

\begin{proof}
    \review{The first part of the Theorem is just a re-statement of Theorem~\ref{theorem:pod}} applied to $H=L^{2}(\Omega)$. The latter, in fact, yields the existence of suitable measurable maps $V_i:\gspace\to L^{2}(\Omega)$
    such that $\int_{\Omega}V_i(\gparam)V_j(\gparam)=\delta_{i,j}$ and
    $$\mathbb{E}\|X_\gparam - \sum_{i=1}^{n}\left(\int_{\Omega}X_\gparam(z)V_i(\gparam)(z)dz\right)V_i(\gparam)\| ^2_{L^{2}(\Omega)}=\inf_{Z\in Q_n}\;\mathbb{E}\int_{\Omega}|X_\gparam(z)-Z(z)|^2dz.$$
    Then, letting 
    $\eta_{\gparam,i}=\int_{\Omega}X_\gparam(z)V_i(\gparam)(z)dz,$
    yields Eq. \eqref{eq:param-kkl}.    
    The only caveat concerns defining the maps $v_i:\gspace\times\Omega\to\mathbb{R}$ such that $v_i(\gparam,\cdot)=V_i(\gparam)$ almost everywhere in $\Omega$. In fact, pointwise evaluations of $L^{2}$ functions are a delicate matter, and we cannot simply set $v_i(\gparam,z)=V_i(\gparam)(z).$ However, we can easily circumvent this problem by \review{fixing a suitable orthonormal basis $\{e_k\}_{k=1}^{+\infty}$ of $L^2(\Omega)$, with a precise choice of the representatives $e_k:\Omega\to\mathbb{R},$ and then setting
    $$v_i(\gparam, x):=\sum_{k=1}^{+\infty}\langle V_i(\gparam), e_k\rangle_{L^2(\Omega)} e_k(x).$$
    In the same spirit, we notice that \emph{(i)} is just a re-statement of Theorem~\ref{theorem:continuous-pod}.}

    \review{For what concerns \emph{(ii)}, instead, let $\lambda_i^\gparam$ denote the $i$th eigenvalue of $K_\gparam$ and let $\tilde{v}_i^\gparam$ be a corresponding eigenfunction. Then
    $$\tilde{v}_i^\gparam(z)=\frac{1}{\lambda_i^\gparam}\int_{\Omega}K_\gparam(z,y)\tilde{v}_i(y)^\gparam dy.$$
    a.e. in $\Omega$. In particular, 
    $|\tilde{v}_i^\gparam(z)-\tilde{v}_i^\gparam(z')|\le\frac{1}{\lambda_i^\gparam}\int_{\Omega}|K_\gparam(z,y)-K_\gparam(z',y)||\tilde{v}_i(y)|dy\le L|\Omega|^{1/2}(\lambda_i^\gparam)^{-1}|z-z'|$ for all $z,z'\in\Omega$, where we used the Lipschitz property of $K_\gparam$ and the orthonormality of $\tilde{v}_i^\gparam$ in $L^2(\Omega)$ combined with Hölder's inequality. Thus, for all $\gparam\in\gspace,$ the subspace spanned by the eigenfunctions associated to $\{\lambda_i^\gparam\}_{i=1}^{n}$ consists of Lipschitz functions with a uniformly bounded Lipschitz constant, the bound being
    $$L':=\frac{L|\Omega|^{1/2}}{\min_{\gparam\gspace}\lambda_i^\gparam}<+\infty.$$
    Notice that the above is finite thanks to the gap assumption, the compactness of $\gspace$ and the continuity of the eigenvalues with respect to $\gparam.$ Since, by construction,
    $$v_i(\gparam,\cdot)\in\text{span}\{\tilde{v}_i^\gparam\}_{i=1}^{n},$$
     the $v_i(\gparam,\cdot)$'s are guaranteed to be Lipschitz in $z$, with a Lipschitz constant no larger than $nL'$.}

    \review{Next, to see that Eq. \eqref{eq:sup-continuity} holds, fix any $\gparam\in\gspace$ and, seeking contraddiction, assume there exists some $\varepsilon>0$ and a suitable sequence $\gparam_k\to\gparam$ such that $\sup_{z\in\Omega}|v_i(\gparam,z)-v_i(\gparam_k,z)|>\varepsilon$ for all $k.$ Thanks to the previous observation, the sequence $\{v_i(\gparam_k,\cdot)\}_{k=1}^{+\infty}$ is equicontinuous and equibounded. In particular, by the Arzelà–Ascoli Theorem, there exists a subsequence $\{v_{i}(\gparam_{k_j},\cdot)\}_j$ converging to a suitable limit in the supremum norm. However, by \emph{(ii)}, we know that $v_{i}(\gparam_{k_j},\cdot)\to v_{i}(\gparam,\cdot)$ in the $L^2$ sense. Then, by uniqueness of the limit, $v_{i}(\gparam_{k_j},\cdot)\to v_{i}(\gparam,\cdot)$ in supremum norm: contraddiction.}

    \review{The joint continuity in $(\gparam,z)$ is now trivial since, for every $\gparam,\gparam'\in\gspace$ and every $z,z'\in\Omega$ we have
    \begin{multline*}
        |v_i(\gparam,z)-v_i(\gparam',z')|\le|
        v_i(\gparam,z)-v_i(\gparam',z)|+
        |v_i(\gparam',z)-v_i(\gparam',z')|\\\le \sup_{z\in \Omega}|
        v_i(\gparam,z)-v_i(\gparam',z)| + nL'|z-z'|,     
    \end{multline*}
    which vanishes whenever $\gparam'\to\gparam$ and $z'\to z.$}
\end{proof}

\label{sec:corollaries}
\label{SEC:COROLLARIES}

\section{\review{Implications for deep learning models serving as surrogates for parametric low-rank approximation}}

\review{In this Section, we aim to apply the theory developed so far to a recently emerging class of techniques for parametric low-rank approximation, that is, methods based on deep neural network architectures; see, e.g., \cite{berman2024colora, franco2024deep, brivio2024handling}. For the sake of simplicity, in what follows, we shall assume all Hilbert spaces to be finite-dimensional, $H_1=H_2=H=\mathbb{R}^N$, and the parameter space to be of the form $\gspace=[0,1]^p$. We shall start with a brief introduction to the topic, followed by a presentation of two Corollaries directly descending from our analysis in Section~\ref{sec:results}. Lastly, we will present some numerical experiments illustrating how the theory translates into practice.}
\\\\
\review{To put this into context, assume we are given a parametrized matrix, $\mathbf{A}_\gparam\in\mathbb{R}^{N\times N}$. We are interested in computing a low-rank approximation of $\mathbf{A}_\gparam$ for varying $\gparam\in\gspace.$ As we mentioned in the Introduction, in principle, this could be achieved by running an SVD for each $\gparam\in\gspace$, and computing $\mathbf{A}_\gparam^n$ by truncating the expansion. However, this approach is ---in general--- computationally unfeasible, as it requires solving an expensive singular value problem multiple times. For this reason, one might be interested in a surrogate model, $\tilde{\mathbf{A}}:\gspace\to\mathbb{R}^{N\times N}$ constructed such that $$\rank(\tilde{\mathbf{A}}(\gparam))\le n\quad\text{and}\quad\tilde{\mathbf{A}}(\gparam)\approx \mathbf{A}_\gparam.$$}

\review{In practice, dealing with the large dimension at output, $N^2$, and guaranteeing that the model only returns matrices of rank no larger than $n$, is not immediate. These issues however, can be addressed by reformulating the learning problem as the search for two maps, $\tilde{\mathbf{U}},\tilde{\mathbf{V}}:\gspace\to\mathbb{R}^{N\times n}$, such that
\begin{equation}
    \label{eq:nnapprox}
    \mathbf{A}_\gparam\approx \tilde{\mathbf{V}}(\gparam)\tilde{\mathbf{U}}^\top(\gparam).
\end{equation}
Clearly, other parametrizations are also possible. In the case of deep learning surrogates, the matrix-valued maps $\tilde{\mathbf{U}}$ and $\tilde{\mathbf{V}}$ are constructed using suitable neural network architectures. Here, a map from $\tilde{\mathbf{B}}:[0,1]^p\to\mathbb{R}^{N\times n}$ is said to be a neural network with activation function $\rho:\mathbb{R}\to\mathbb{R}$ if its action can be written as
\begin{equation}
    \label{eq:nn}
    \tilde{\mathbf{B}}(\gparam)=(R\circ T_\ell \circ \rho \circ T_{\ell-1}\ldots\circ \rho\circ T_1)(\gparam),
\end{equation}
where $\ell\in\mathbb{N}$, each $T_j$ is an affine transformation from $\mathbb{R}^{m_{j-1}}$ to $\mathbb{R}^{m_{j}}$, $m_0=p$, $m_{\ell}=Nn$, and $R:\mathbb{R}^{Nn}\to\mathbb{R}^{N\times n}$ is a reshape operator mapping a vector of length $Nn$ into a matrix of shape $N\times n$, with entries being filled row-wise. Notice also that, here, with little abuse of notation, the scalar map $\rho$ is being applied componentwise on vectors.
}

\review{For fixed $\rho$, let us denote by $\mathscr{N}_\rho(p,N,n)$ the subset of all neural network architectures from $[0,1]^p$ to $\mathbb{R}^{N\times n}$, obtained by varying $\ell$ and the affine operators $T_j$ in Eq. \eqref{eq:nn}. The learning problem thus consists in seeking for two maps $$\tilde{\mathbf{U}},\tilde{\mathbf{V}}\in\mathscr{N}_\rho(p,N,n),$$ such that Eq. \eqref{eq:nnapprox} holds (in a suitable sense). The notable thing about using neural networks is that, if $\rho$ is continuous but not a polynomial, then the set $\mathscr{N}_\rho(p,N,n)$ is dense in $C(\gspace;\mathbb{R}^{N\times n})$ and in $L^q_{\mu}(\gspace;\mathbb{R}^{N\times n})$ for all $q\in[1,+\infty)$ and all finite measures $\mu$ over $\gspace.$ This is a consequence of the well-known universal approximation theorems: see, e.g., \cite{hornik1991approximation}. Consequently, a deep learning surrogate can become as accurate as any parametric low-rank approximant that is available in those two classes. In particular, if suitable conditions are satisfied, a method of this kind can reach levels of accuracy that are arbitrarily close to the ones of parametric SVD --- thus reaching the ideal optimum. Most of the times, such a convergence will be possible in an $L^p_\mu$-sense, coherently with the measurability of parametric SVD in Theorem~\ref{theorem:svd}. In contrast, we expect a uniform convergence over $\gparam\in\gspace$ to be possible only for those problems which admit a continuous version of the parametric SVD ---specifically, when the gap condition in Theorem~\ref{theorem:continuous-svd} holds.}

\review{In general, the same ideas can be transferred to the case of parametric low-rank approximation of random variables. We formalize these facts in the two Corollaries right below. Given the simplicity of their derivation, we only report the proof of the first one. In the next subsection, instead, we shall illustrate how these theoretical results can effectively be observed in practice.}\\

\begin{corollary}
    \label{corollary:parametric-svd}
    \review{Let $\gspace=[0,1]^p$ be equipped with some finite measure $\mu$. Let $\gspace\ni\gparam\to \mathbf{A}_\gparam\in\mathbb{R}^{N\times N}$ be a measurable map such that $\int_{\gspace}\opnorm{\mathbf{A}_\gparam}^2\mu(d\gparam)<+\infty.$
    Fix any integer $0<n\le N$ and any continuous nonpolynomial activation function $\rho$. Then, for every $\varepsilon>0$, there exist $\tilde{\mathbf{U}}_{n}, \tilde{\mathbf{V}}_{n}\in\mathscr{N}_\rho(p,N,n)$ such that
        \begin{equation}
        \label{eq:nnsvd1}\int_{\gspace}\opnorm{\mathbf{A}_\gparam-\tilde{\mathbf{V}}_{n}(\gparam)\tilde{\mathbf{U}}_{n}^\top(\gparam)}\mu(d\gparam)<\varepsilon+\int_{\gspace}\sigma_{n+1}(\mathbf{A}_{\gparam})\mu(d\gparam).\end{equation}
    %and
    %$$\int_{\gspace}\hsnorm{\mathbf{A}_\gparam-\tilde{\mathbf{V}}_{n}(\gparam)\tilde{\mathbf{U}}_{n}^\top(\gparam)}\mu(d\gparam)<\varepsilon+\int_{\gspace}{\left[\;\sum_{j=n+1}^{N}\sigma_{j}(\mathbf{A}_{\gparam})^2\right]^{1/2}\mu(d\gparam)}.$$
    Additionally, if $\mathbf{A}_\gparam$ depends continuously on $\gparam$ and $\sigma_{n}(\mathbf{A}_\gparam)>\sigma_{n+1}(\mathbf{A}_\gparam)$ for all $\gparam\in\gspace$, then the two networks, $\tilde{\mathbf{U}}_{n}$ and $ \tilde{\mathbf{V}}_{n}$, can be chosen such that
     \begin{equation}
         \label{eq:nnsvd2}
         \sup_{\gparam\in\gspace}\opnorm{\mathbf{A}_\gparam-\tilde{\mathbf{V}}_{n}(\gparam)\tilde{\mathbf{U}}_{n}^\top(\gparam)}<\varepsilon+\sup_{\gparam\in\gspace}\;\sigma_{n+1}(\mathbf{A}_{\gparam}).
     \end{equation}
    %and
    %$$\sup_{\gparam\in\gspace}\hsnorm{\mathbf{A}_\gparam-\tilde{\mathbf{V}}_{n}(\gparam)\tilde{\mathbf{U}}_{n}^\top(\gparam)}<\varepsilon+\sup_{\gparam\in\gspace}{\left[\;\sum_{j=n+1}^{N}\sigma_{j}(\mathbf{A}_{\gparam})^2\right]^{1/2}}.$$
    Finally, the same holds true if, in all of the above, the operator norm is substituted with the Hilbert-Schmidt norm and $\sigma_{n+1}(\mathbf{A})_\gparam$ is replaced by $\left[\;\sum_{j=n+1}^{N}\sigma_{j}(\mathbf{A}_{\gparam})^2\right]^{1/2}.$
    }
\end{corollary}
\begin{proof}
    \review{Up to adapting the notation, we notice that Theorem~\ref{theorem:svd} guarantees the existence of measurable maps $\mathbf{U}_n,\mathbf{V}_n:\gspace\to\mathbb{R}^{N\times n}$ and $\boldsymbol{\Sigma}_n:\gspace\to\mathbb{R}^{n\times n}$ such that
    $$\opnorm{\mathbf{A}_\gparam-\mathbf{V}_{n}(\gparam)\boldsymbol{\Sigma}_n(\gparam)\mathbf{U}_{n}^\top(\gparam)}<\sigma_{n+1}(\mathbf{A}_{\gparam})\quad\forall\gparam\in\gspace.$$
    Define $\hat{\mathbf{U}}_n(\gparam):=\mathbf{U}_n(\gparam)\boldsymbol{\Sigma}_n(\gparam).$ Since $\mathbf{V}_n$ and $\mathbf{U}_n$ take values in the set of orthonormal matrices, whereas $\boldsymbol{\Sigma}_n$ contains the first $n$ singular values of $\mathbf{A}_\gparam$, we have, pointwise in the parameter space,
    $$\opnorm{\mathbf{V}_n(\gparam)}\le1\quad\text{and}\quad\opnorm{\hat{\mathbf{U}}_n}\le\opnorm{\boldsymbol{\Sigma}_n(\gparam)}=\sigma_{1}(\mathbf{A}_\gparam)=\opnorm{\mathbf{A}_\gparam}.$$
    In particular, by hypothesis, all of the above are square-integrable over $\gparam\in\gspace$, meaning that $\mathbf{V}_n,\hat{\mathbf{U}}_n\in L^2(\gspace;\mathbb{R}^{N\times N})$, the co-domain of the Bochner space being equipped with the operator norm. In light of this, consider the functional $J:L^2(\gspace;\mathbb{R}^{N\times n})\times L^2(\gspace;\mathbb{R}^{N\times n})\to\mathbb{R}$ defined as
    $$J(\mathbf{U},\mathbf{V}):=\int_\gspace\opnorm{\mathbf{A}_\gparam-\mathbf{V}(\gparam)\mathbf{U}^\top(\gparam)}\mu(d\gparam).$$ Since $\mathbf{A}=\mathbf{A}_\gparam$ is in $L^2(\gspace;\mathbb{R}^{N\times N})\subset L^1(\gspace;\mathbb{R}^{N\times N})$ and $\mathbf{V},\mathbf{U}\in L^2(\gspace;\mathbb{R}^{N\times n})$ implies $\mathbf{V}\mathbf{U}^\top\in\mathbf{V}_n,\hat{\mathbf{U}}_n\in L^1(\gspace;\mathbb{R}^{N\times N})$ ---by Hölder's inequality---, it is straightforward to see that $J$ is continuous. Thus, by density,
    $$\inf_{\tilde{\mathbf{U}},\tilde{\mathbf{V}}\in\mathscr{N}_\rho(p,N,n)}J(\tilde{\mathbf{U}},\tilde{\mathbf{V}}) = \inf_{\mathbf{U},\mathbf{V}\in L^2(\gspace;\mathbb{R}^{N\times n})}J(\mathbf{U},\mathbf{V})=J(\hat{\mathbf{U}}_n,\mathbf{V}_n)=\int_{\gspace}\sigma_{n+1}(\mathbf{A}_\gparam)\mu(d\gparam).$$
    Eq. \eqref{eq:nnsvd1} now follows easily. The other inequality, Eq. \eqref{eq:nnsvd2}, instead, can be proven analogously using the functional $\tilde{J}(\mathbf{U},\mathbf{V}):=\sup_{\gparam\in\gspace}\opnormsmall{\mathbf{A}_\gparam-\mathbf{V}(\gparam)\mathbf{U}_\gparam^\top}$ and the density of $\mathscr{N}_\rho(p,N,n)$ in $C(\gspace;\mathbb{R}^{N\times n})$. Notice that this is only possible thanks to the gap and the continuity assumptions, which, by Theorem~\ref{theorem:continuous-svd} guarantee the existence of a suitable optimum in $C(\gspace;\mathbb{R}^{N\times n}).$ Finally, in both cases, the Hilbert-Schmidt case follows similarly.}
\end{proof}

\begin{corollary}
\label{corollary:parametric-pod}
 \review{Let $\gspace=[0,1]^p$ be equipped with some finite measure $\mu$. Let $\{X_\gparam\}_{\gparam\in\gspace}$ be a family of random variables in $\mathbb{R}^{N}$ with finite second moment.} Assume that the map $\xi\to X_\gparam$ is measurable and
$\int_{\gspace}\mathbb{E}^{1/2}\|X_\gparam\|^2\mu(d\gparam)<+\infty.$
For every $\gparam\in \gspace$, let $\lambda_{1}^\gparam\ge\dots\ge\lambda_{N}^\gparam\ge0$ be the eigenvalues of the uncentered covariance matrix
$\mathbf{C}_{\gparam}:=(\mathbb{E}[X_\gparam^iX_\gparam^j])_{i,j=1}^{N}.$
Fix any integer $0<n\le N$ and a \review{continuous nonpolynomial activation function $\rho.$} Then, for every $\varepsilon>0$ there exists some 
$\tilde{\mathbf{V}}_n\in\review{\mathscr{N}_\rho(p,N,n)}$ such that
    \begin{equation}
    \label{eq:pod-int-gap}\review{\int_{\gspace}\mathbb{E}^{1/2}\|X_\gparam-\tilde{\mathbf{V}}_n(\gparam)\tilde{\mathbf{V}}_n^\top(\gparam)X_\gparam\|^2\mu(d\gparam)< \varepsilon + \int_{\gspace}\left[\;\sum_{i=n+1}^{N}\lambda_{i}^\gparam\right]^{1/2}\mu(d\gparam).}\end{equation}
\review{Additionally, if $\lambda_{n}^\gparam>\lambda_{n+1}^\gparam$ for every $\gparam\in\gspace$ and $\mathbb{E}\|X_\gparam-X_{\gparam'}\|^2\to0$ whenever $\gparam'\to\gparam$, then the network $\tilde{\mathbf{V}}_n$ can be constructed such that
\begin{equation}
    \label{eq:pod-sup-gap}
\sup_{\gparam\in\gspace}\;\mathbb{E}^{1/2}\|X_\gparam-\tilde{\mathbf{V}}_n(\gparam)\tilde{\mathbf{V}}_n^\top(\gparam)X_\gparam\|^2< \varepsilon + \sup_{\gparam\in\gspace}\left[\;\sum_{i=n+1}^{N}\lambda_{i}^\gparam\right]^{1/2}.\end{equation}}
\end{corollary}

\begin{proof}
    \review{This is just a straightforward application of Theorems~\ref{theorem:pod}-\ref{theorem:continuous-pod}. Specifically, the proof can be derived following the same ideas presented in the proof of Corollary~\ref{corollary:parametric-svd}.}
\end{proof}

\begin{remark}
    \review{Clearly, several of the  assumptions in Corollaries~\ref{corollary:parametric-svd}-\ref{corollary:parametric-pod} can be relaxed. For instance, $\gspace$ can be any compact contractible domain, the latter property only being required when considering worst case metrics (supremum norm). Additionally, in Corollary~\ref{corollary:parametric-svd}, the result in Eq. \eqref{eq:nnsvd1} can be easily adapted to the case in which one uses three networks, $\tilde{\mathbf{V}}_n,\tilde{\boldsymbol{\Sigma}}_n,\tilde{\mathbf{U}}_n$, designed to mimic the SVD. Here, however, we refrain from doing so. The reason is that such a representation may not always be possible when considering worst case metrics. This is related to continuity issues, which may arise if the singular values $\sigma_1,\dots,\sigma_n$ cross each other several times. At its core, this is essentially the reason why, in Theorem~\ref{theorem:continuous-svd}, the basis vectors $v_i(\gparam)$ are guaranteed to be orthonormal, whereas the $\tilde{u}_i(\gparam)$ may fail to satisfy such constraint.}\\
\end{remark}

\begin{remark}
\review{Corollary~\ref{corollary:parametric-pod} can be easily generalized to accommodate any norm $\|\cdot\|_{\mathbf{M}}$ defined over $\mathbb{R}^N$ that is compatible with an inner product structure. In fact, if $\|\mathbf{v}\|_{\mathbf{M}}:=\sqrt{\mathbf{v}^\top\mathbf{M}\mathbf{v}}$ for a suitable positive definite matrix $\mathbf{M}=(m_{i,j})_{i,j}$, then one can readily apply Corollary~\ref{corollary:parametric-pod} up to replacing $\mathbf{C}_\gparam$ with$$\mathbf{C}^{\mathbf{M}}_\gparam:=\left(\mathbb{E}\left[\sum_{k=1}^N m_{k,j}X_\gparam^iX_\gparam^k\right]\right)_{i,j=1}^{N}.$$
    In fact, the above is the matrix representation of the operator  $\mathbf{v}\mapsto \mathbb{E}[\scl{\mathbf{v}}{X_\gparam}_{\mathbf{M}}X_\gparam]=\mathbb{E}[\mathbf{v}^\top\mathbf{M}X_\gparam X_\gparam].$
    }
\end{remark}

\subsection{\review{Numerical experiments}}

\review{We present some numerical experiments where we assess the empirical evidence of our findings, specifically focusing on the results in Corollaries~\ref{corollary:parametric-svd}-\ref{corollary:parametric-pod}. To this end, we consider the following model problems and methodologies.}

\review{\paragraph{Problems under study}
We consider four different problems: the first two are devoted to the low-rank approximation of parametrized matrices [\textbf{PM}], whereas the other two discuss the problem of dimensionality reduction for high-dimensional random variables [\textbf{RV}]. In both cases, we provide an example where the gap condition holds and one where it doesn't. For simplicity, in all of the case studies, we set the low-rank dimension to be $n=2.$}

\review{ We remark that all case studies have been handcrafted to analyze the behavior of surrogate models under precise circumstances. In particular, we are not interested in studying the performances of these models in tackling complex real-world applications characterized by elevated computational costs, nor to compare these techniques with other alternatives in the literature. Our sole purpose is to assess whether our findings in Section~\ref{sec:results} are supported by empirical evidence.}
\\\\
\review{\textbf{PM1.} Let $\gspace=[0,1]$. We consider a $4\times 4$ matrix parametrized as
$$\mathbf{A}_\gparam = \left[
\begin{array}{cccc}
\gparam & 0 & 0 & 0 \\
0       & \gparam/2 & 0 & 0 \\
0 & 0 & 1-\gparam & 0 \\
0 & 0 & 0 & 2\gparam^2
\end{array}\right].$$
This case study serves as a representative for parametric problem that do not satisfy the gap condition.}
\\\\
\review{\textbf{PM2.} Let $\gspace=[0,1]$. We consider a $4\times 4$ matrix parametrized as
$\mathbf{A}_\gparam = \left[
\begin{array}{cc}
\mathbf{B}_\gparam & \mathbf{0}^\top\\
\mathbf{0}       & 1
\end{array}\right],$
where $\mathbf{0}=[0,0,0]^\top,$ and $\mathbf{B}_\gparam$ is a $3\times 3$ block given by
$$\mathbf{B}_\gparam=\left[\begin{array}{ccc}
     4 & 0 & 0\\
     0 & 1 + \gparam & 0\\
     0 & 0 & 2 + \gparam
\end{array}\right]\cdot\left(\mathbf{I}_3 + \sin\gparam\cdot\mathbf{K}+(1-\cos\gparam  )\cdot\mathbf{K}^2\right),\quad\text{where}\quad\mathbf{K}=\left[\begin{array}{ccc}
     0 & -1 & 1\\
     1 & 0 & -1\\
     -1 & 1 & 0 
\end{array}\right].$$
In this case, one can easily verify that the gap condition holds true. Intuitively, the parametric block $\mathbf{B}_\gparam$ is being constructed using a parametrized diagonal matrix and a parametric rotation obtained via Rodrigues' rotation formula \cite{rodrigues1840lois}.}
\\\\
\review{\textbf{RV1.} Let $\gspace=[0,1]^2$. Let $V_h$ be the Finite Element space of continuous piecewise-linear maps constructed on top of a uniform grid of stepsize $h=0.02$ defined over $\Omega=(-1,1).$ Let $N:=\dim(V_h)=1001.$ As a state space, we consider the Euclidean space arising from the algebraic representation of $V_h$, that is, $(\mathbb{R}^N,\|\cdot\|_{\mathbf{M}})\cong (V_h,\|\cdot\|_{L^2(\Omega)})$, with $\mathbf{M}$ the Gram matrix associated to the Lagrangian basis of $V_h$ and the $L^2$-inner product.}

\review{Let $\{x_j\}_{j=1}^{N}$ be the nodes of the finite element grid. For every $\gparam=(\gparam^{(1)},\gparam^{(2)})$, we let  $q_\gparam:=\gparam^{(1)}-0.5$ and $\delta_\gparam:=75\gparam^{(2)}+25$. We consider a parameter dependent random variable given by
$$X_\gparam = \left[\sum_{k=1}^{N}10(k+1)^{-2}\eta_k e^{-\delta_\gparam|x_j-q_\gparam|^2}\sin(k\pi(x_j-q_\gparam)/2)\right]_{j=1}^{N}$$
corresponding to the discretization of a suitable random field defined over $\Omega$. Here, $\eta_1,\dots,\eta_{N}$ are independent standard gaussians. This case study serves as a model problem for the scenario in which the family $X_\gparam$ satisfies the gap condition.}
\\\\
\review{\textbf{RV2.} We consider the same setting as in RV1 but parametrize the discrete random field as
$$X_\gparam = \left[1+\sum_{k=1}^{200}\left(\tau_\gparam e^{-\delta_\gparam |7-k|^2} + (1-\tau_\gparam)e^{-\delta_\gparam k^2}\right)\eta_k\sin(k\pi x_j)\right]_{j=1}^{N},$$
where $\delta_\gparam=0.9\gparam^{(1)}+0.1$ and $\tau_\gparam=\gparam^{(2)}$. This time, the gap condition at $n=2$ fails. Essentially the parameters are used to change the dominating frequencies in the random field, which causes crossing of the eigenvalues in the covariance matrix.}

\review{\paragraph{Methods}
For all of the case studies, we construct deep learning-based parametric low-rank approximants of different complexities, monitoring whether the accuracy of these models improves or not when considering more complex architectures. The results in Corollaries~\ref{corollary:parametric-svd}-\ref{corollary:parametric-pod}, in fact, are based on density arguments. In particular, one expects that, by increasing the expressivity of the models and the number of training data, the $\varepsilon$-gap in Eqs. \eqref{eq:nnsvd1}-\eqref{eq:pod-sup-gap} can be made smaller and smaller. To evaluate the quality of the approximation, we consider two different metrics: one based on the supremum norm (worst-case scenario) and one based on the integral norm (average case). Precisely, for problems $\textbf{PM1-2}$, we rely on the following estimators
$$E_{\sup}^{\text{pm}}:=\sup_{k=1,\dots,M}\|\mathbf{A}_{\gparam_k}-\tilde{\mathbf{A}}_n(\gparam_k)\|_{F},\quad\quad E_{\text{avg}}^{\text{pm}}:=\frac{1}{M}\sum_{i=1}^{M}\|\mathbf{A}_{\gparam_k}-\tilde{\mathbf{A}}_n(\gparam_k)\|_{F},$$
with $\|\cdot\|_{F}$ the Frobenius norm ---equivalently, the Hilbert-Schmidt norm.
Both quantities are computed on values of the problem parameter that were not included in the training set, sampled randomly and uniformly across $\gspace$. Here, we use $M=1001.$ For problems $\textbf{RV1-2}$, instead, we consider the following metric,
$$E_{\sup}^{\text{rv}}:=\sup_{k=1,\dots,M}\left[\frac{1}{S}\sum_{l=1}^{S}\|X_{\gparam_k}^{(l)}-\tilde{\mathbf{V}}(\gparam_k)\tilde{\mathbf{V}}(\gparam_k)X_{\gparam_k}^{(l)}\|\right]^{1/2},$$
and $$E_{\text{avg}}^{\text{rv}}:=\frac{1}{M}\sum_{k=1}^{M}\left[\frac{1}{S}\sum_{l=1}^{S}\|X_{\gparam_k}^{(l)}-\tilde{\mathbf{V}}(\gparam_k)\tilde{\mathbf{V}}(\gparam_k)X_{\gparam_k}^{(l)}\|\right]^{1/2},$$where, for each $k=1,\dots, M$, $X^{(1)}_{\gparam_k},\dots,X^{(S)}_{\gparam_k}$ is an i.i.d. random sample of size $S$ drawn from $X_{\gparam_k}.$ The two formulae are obtained by approximating the expected values in \eqref{eq:pod-int-gap}-\eqref{eq:pod-sup-gap} via plain Monte Carlo. In our experiments we use $M=200$ and $S=50$}.

\review{We construct the low-rank approximants following the ideas presented in Corollaries~\ref{corollary:parametric-svd}-\ref{corollary:parametric-pod}. Specifically, for problems $\textbf{PM1-2}$, we set $\tilde{\mathbf{A}}_n(\gparam)=\tilde{\mathbf{V}}(\gparam)\tilde{\mathbf{U}}^\top(\gparam)$ where $\tilde{\mathbf{V}},\tilde{\mathbf{U}}:\gspace\to\mathbb{R}^{N\times n}$ are matrix-valued deep neural networks obtained by minimizing the following loss function,
$$\mathscr{L}(\tilde{\mathbf{U}}, \tilde{\mathbf{V}}):=\frac{1}{L}\sum_{j=1}^{L}\|\mathbf{A}_{\tilde{\gparam}_j}-\tilde{\mathbf{V}}(\tilde{\gparam}_j)\tilde{\mathbf{U}}^\top(\tilde{\gparam}_j)\|_{F}^2.$$
Here, $\tilde{\gparam}_{1},\dots,\tilde{\gparam}_{L}$ are suitable training points, drawn randomly and uniformly across $\gspace.$}

\review{As we anticipated, in order to investigate whether some kind of convergence is happening, we are interested in repeating the same experiment for varying architectures and different sample sizes. To accommodate this fact, we propose setting the number of training points $L$ equal to the total number of trainable parameters in the deep learning model. This is to ensure that more complex architectures are trained on proportionally more data, helping to reduce the risk of overfitting.}

\review{For problems $\textbf{RV1-2}$ we proceed similarly but train the matrix-valued network $\tilde{\mathbf{V}}:\gspace\to\mathbb{R}^{N\times n}$ by minimizing the loss function below
$$\mathscr{L}(\tilde{\mathbf{V}}):=\frac{1}{L}\sum_{j=1}^{L}\|\tilde{X}_{\tilde{\gparam}_k} - \tilde{\mathbf{V}}(\tilde{\gparam}_k)\tilde{\mathbf{V}}^\top(\tilde{\gparam}_k)\tilde{X}_{\tilde{\gparam}_k}\|^2,$$
with each $\tilde{X}_{\tilde{\gparam}_k}$ being a random realization of $X_{\tilde{\gparam}_k}.$}

\review{All case studies were implemented in Python using the \textit{pytorch} library \cite{paszke2019pytorch} and the \textit{dlroms} package \cite{franco2024dlroms}. Experimental data and source codes are available upon request to the author.}

\review{\paragraph{Results}
Our results are summarized in Figures~\ref{fig:svd1}-\ref{fig:pod2}. There, we compare the accuracy of the deep learning models against the ideal lower-bound given by parametric SVD and parametric POD, respectively. In particular, we do not report $E_{\sup}^{\text{pm}}$, $E_{\text{avg}}^{\text{pm}}$, $E_{\sup}^{\text{rv}}$ and $E_{\text{rv}}^{\text{pm}}$ directly, but first subtract the ideal error to ease inspection. Intuitively, such discrepancy can be thought as a quantification of the constant $\varepsilon>0$ appearing in Corollaries~\ref{corollary:parametric-svd}-\ref{corollary:parametric-pod}.}

\begin{figure}
    \centering
    \includegraphics[width=0.99\linewidth]{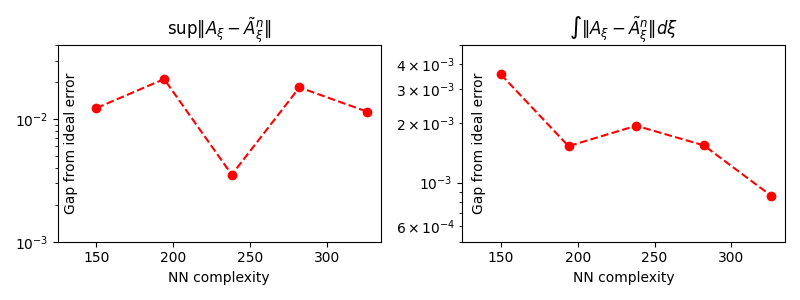}
    \includegraphics[width=0.99\linewidth]{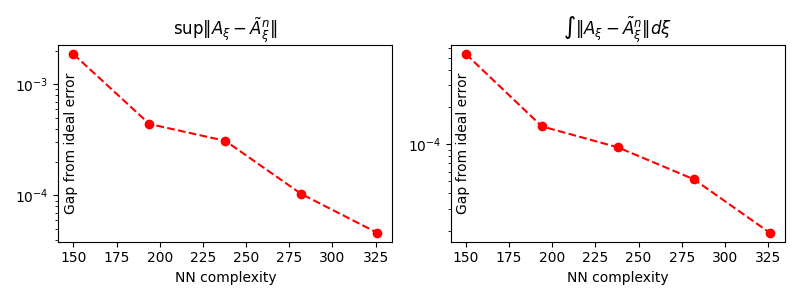}
    \caption{\review{Parametric low-rank approximation of a matrix. Error analysis for problems \textbf{PM1} (top) and \textbf{PM2} (bottom). $x$-axis: complexity of the overall neural network architecture, defined as the total number of trainable parameters in the model. $y$-axis: difference between the error of the model ($E_{\sup}^{\text{pm}}$ on the left, $E_{\text{avg}}^{\text{pm}}$ on the right) and the ideal error provided by parametric SVD. The latter were computed by estimating the Hilbert-Schmidt equivalent of the terms at the right-hand-side of \eqref{eq:nnsvd2} and \eqref{eq:nnsvd1} with $\varepsilon=0.$}}
    \label{fig:svd1}
\end{figure}

\begin{figure}
    \centering
    \includegraphics[width=0.99\linewidth]{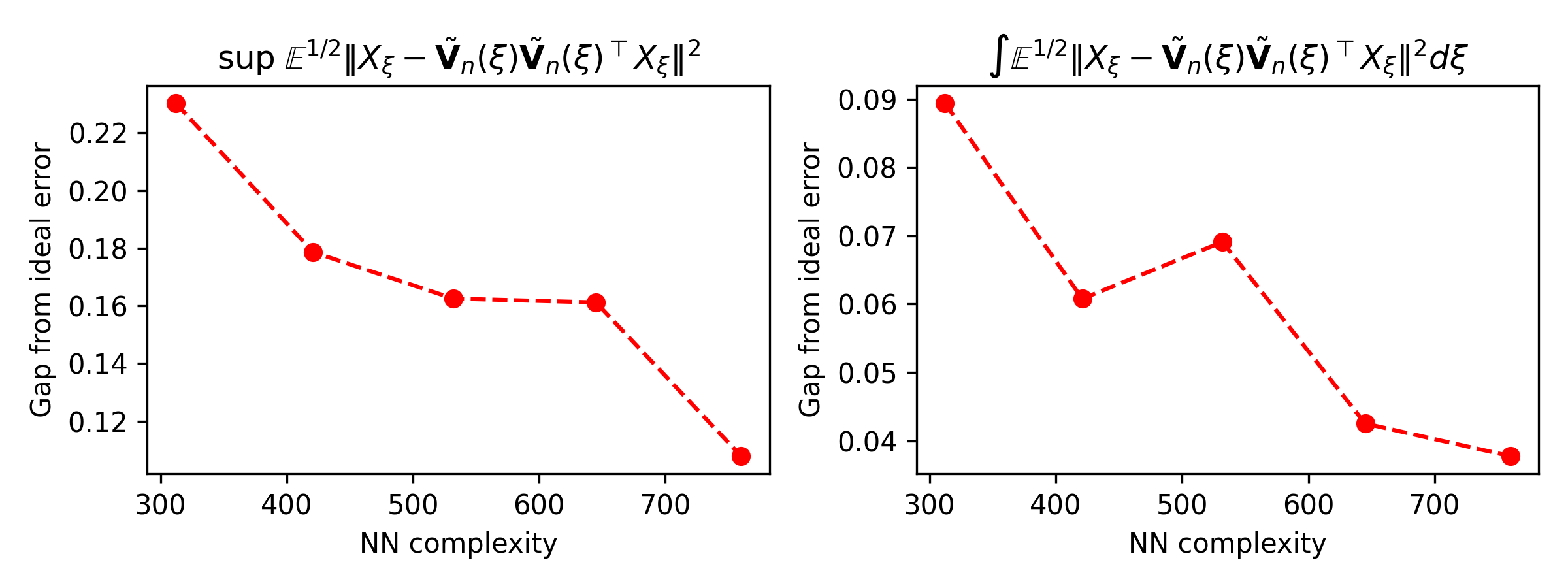}
    \includegraphics[width=0.99\linewidth]{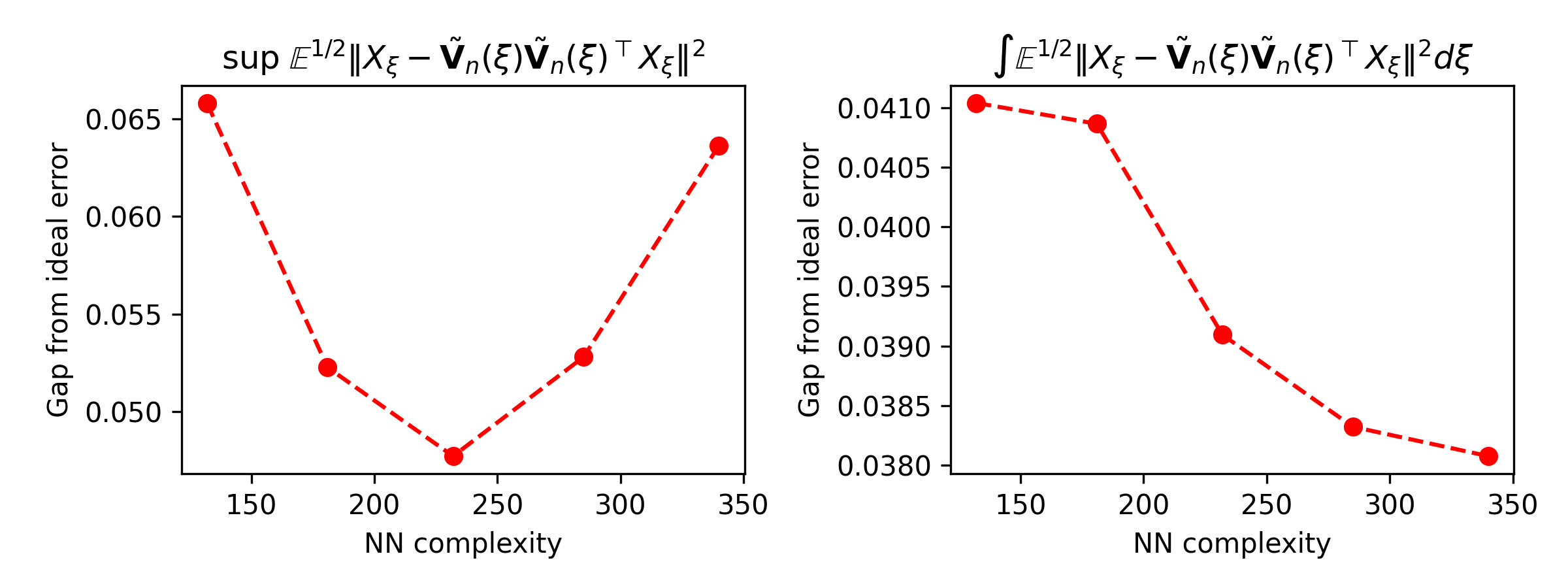}
    \caption{\review{Parametric low-rank approximation of a discretized random field. Error analysis for problems \textbf{RV1} (top) and \textbf{RV2} (bottom). $x$-axis: complexity of the overall neural network architecture, defined as the total number of trainable parameters in the model. $y$-axis: difference between the error of the model ($E_{\sup}^{\text{rv}}$ on the left, $E_{\text{avg}}^{\text{rv}}$ on the right) and the ideal error provided by parametric POD. The latter were computed by estimating the terms at the right-hand-side of \eqref{eq:pod-sup-gap} and \eqref{eq:pod-int-gap} with $\varepsilon=0.$}}
    \label{fig:pod2}
\end{figure}

\review{%To start, we notice that the gap between the actual error and the ideal one is always nonnegative, coherently with the well-established optimality of parametric SVD and parametric POD. In general, this gap tends to decrease as one considers more complex architectures and adds more training data, suggesting a convergence of the deep learning models towards the ideal optimum. However, this is only true when considering average metrics. According to worst-case metrics, in fact, the improvement of the surrogate models is only observed in two out of four cases.}
We begin by observing that the gap between the actual error and the ideal error is always nonnegative, consistent with the known optimality of parametric SVD and POD. In general, this gap tends to shrink with increasing model complexity and additional training data, suggesting convergence of the deep learning surrogates toward the ideal solution. However, this trend holds primarily for average-case metrics: when considering worst-case performance, instead, a clear improvement is observed in only two out of four scenarios.}

\review{Notably, this outcome is not unexpected and aligns well with Corollaries~\ref{corollary:parametric-svd}–\ref{corollary:parametric-pod}. Specifically, we notice that in those cases where the spectral gap condition is satisfied (\textbf{PM2} and \textbf{RV1}), the surrogates improve under both average and worst-case metrics. Conversely, when the spectral condition fails (\textbf{PM1} and \textbf{RV2}), convergence toward the parametric SVD and POD is seen only in average metrics.}

\review{Intuitively, this is related to the fact that neural network surrogates are inherently continuous with respect to $\gparam$, whereas parametric SVD and POD are not necessarily so. As a result, one can expect convergence to happen in the integral sense, but uniform convergence is not guaranteed. Indeed, since the space $C(\gspace;\mathbb{R}^{N\times n})$ is closed under the supremum norm, a sequence of neural networks will never be able to approximate a discontinuous target uniformly.}

\review{We conclude with a brief note on computational costs. While problems \textbf{PM1-2} were intended as simple toy examples and are thus not particularly relevant from a computational standpoint, problems \textbf{RV1-2} more closely reflect the motivating framework outlined in the Introduction and at the beginning of this Section. For instance, we mention that in \textbf{RV1} the deep learning surrogate offered a substantial speed-up over classical POD, requiring 0.13 seconds per 4000 evaluations in place of 10.2 seconds.}

%\review{This is not entirely surprising. In fact, our findings are perfectly aligned with the results discussed in Corollaries~\ref{corollary:parametric-svd}-\ref{corollary:parametric-pod}. In fact: when the spectral gap condition holds (\textbf{PM2} and \textbf{RV1}), the surrogate models reported improved performances in both the average and the worst-case metrics; when the spectral condition is not satisfied, instead, (\textbf{PM1} and \textbf{RV2}), the models seem to converge towards parametric SVD and parametric POD only in the average metric. Intuitively, this issue is posed by the fact that, while all deep learning surrogates are intrinsically continuous with respect to $\gparam$, parametric SVD and parametric POD are not always so. Consequently, one can reasonably expect a convergence in the integral-sense, but not in the supremum norm. In this respect, recall that $C(\gspace;\mathbb{R}^{N\times n})$ is closed in the supremum norm: thus, a sequence of neural networks will never be able to approximate uniformly a discontinuous map arbitrarily well.}

\section{Conclusions}
\label{sec:conclusions}
In this work, we presented a unified framework for parametric low-rank approximation, covering applications ranging from numerical analysis ---such as the approximation of linear operators--- to probability and statistics ---such as the dimensionality reduction of Hilbert-valued random variables. Additionally, we established foundational results regarding the regularity of parametric algorithms, including parametric SVD and parametric POD, in terms of measurability and continuity, with implications for adaptive and conditional algorithms relying upon universal approximators such as deep neural networks. \review{Numerical experiments were conducted to validate these findings. We showed that, while deep learning models can be used to surrogate parametric low-rank approximation algorithms, the quality of the approximation depends on the underlying problem. Specifically, even though in most cases one can achieve nearly optimal performances according to \textit{average} metrics, a more careful inspection is needed to ensure that the proposed low-rank approximants are uniformly accurate across the whole parameter space. In those cases, a uniform spectral gap condition turns out to be sufficient. Since the latter depends on the chosen low-rank dimension, we speculate that these considerations could eventually lead to novel strategies for the choice of the truncation index $n$.}

Our results, derived under extremely mild assumptions, are nearly as general as possible. This distinguishes our analysis from other fields, which typically focus on small parametric variations and specific regimes where singularities and discontinuities are less likely to occur.
We framed our discussion within the context of separable Hilbert spaces, which makes our analysis applicable to both finite and infinite-dimensional problems. Extending our theory to separable Banach spaces is of interest but not straightforward due to the absence of a canonical form for compact operators ---although they share the notion of singular values, or $s$-numbers \cite{carl1981entropy}. Furthermore, this extension will likely require different analytical tools, as our current construction heavily relies on the certain properties that are specific to Hilbert spaces, such as reflexivity and the existence of a pre-dual. % (to exploit weak*-compactness).% and the Radon-Riesz property.

Another intriguing research direction is to explore higher regularity properties, such as Lipschitz continuity and differentiability.
From an applicative standpoint, this could bring new insights on adaptive algorithms relying upon universal approximators, allowing, for instance, to discuss the computational complexity of the approximating algorithms, rather than their existence alone, in analogy to the developments in \cite{hornik1991approximation}-\cite{yarotsky2017error}.

\section*{Declarations}

\textbf{Acknowledgements} The present research is part of the activities of project "Dipartimento di Eccellenza 2023-2027", funded by MUR, and of "Gruppo Nazionale per il Calcolo Scientifico" (GNCS), of which the author is member. The author would also like to thank Simone Brivio for inspiring the development of this work, Prof. Paolo Zunino for his precious support, and  Dr. Jacopo Somaglia for the insightful discussions that indirectly contributed to this research.\\\\
\textbf{Funding} The author has received funding under the project \textsf{Cal.Hub.Ria}, Piano Operativo Salute, traiettoria 4, funded by MSAL, and under the project \textsf{Reduced Order Modeling and Deep Learning for the real-time approximation of PDEs (DREAM)}, grant no. FIS00003154, funded by the Italian Science Fund (FIS) and by Ministero dell'Università e della Ricerca (MUR).
\\\\
\textbf{Competing interests} The author has no competing interests to declare that are relevant to the content of this article. 

\bibliography{sn-bibliography}

\begin{appendices}
\renewcommand{\thelemma}{\Alph{section}.\arabic{lemma}}
\setcounter{lemma}{0}
\renewcommand{\thecorollary}{\Alph{section}.\arabic{corollary}}
\setcounter{corollary}{0}

\newcommand{\customParagraph}[1]{\;\newline\textit{\textbf{#1}}}

\section{Auxiliary results}
\label{appendix:auxiliary}

\begin{lemma}[Truncation Lemma]
    \label{lemma:truncation}
    Let $n\in\mathbb{N}_+$. Let $\{\lambda_i\}_{i=1}^{+\infty}\in\ell^1(\mathbb{N}_+)$ be such that $$\lambda_1\ge\dots\ge\lambda_{n-1}\ge\lambda_n = \dots = \lambda_{n+r} > \lambda_{n+r+1},$$
    for some $r\ge0$. Let $\mathscr{A}=\{\{a_i\}_{i=1}^{\infty}\;:\;0\le a_i\le1,\;\sum_{i=1}^{+\infty}a_i=n\}\subset\ell^{\infty}(\mathbb{N}_+)$, and consider the linear functional
    $l:\ell^{\infty}(\mathbb{N}_+)\to\mathbb{R}$,
    $$l(a)=\sum_{i=1}^{+\infty}\lambda_ia_i.$$
    All maximizers $a^*\in \mathscr{A}$ satisfy $\sum_{i=1}^{n+r}a_i^*=n$ and $l(a^*)=\sum_{i=1}^{n}\lambda_i.$ In particular, if $r=0$, then $l$ admits a unique maximizer within $\mathscr{A}$, obtained by setting $a_i=1$ for all $i=1,\dots,n$ and $a_i=0$ for all $i>n.$
\end{lemma}

\begin{proof}
     Notice that $A$ is both closed and convex. Consequently, since $\ell^{\infty}(\mathbb{N}_+)$ is the topological dual of $\ell^{1}(\mathbb{N}_+)$, classical arguments of weak$^*$-compactness show that $l$ admits one or more maxima over $A$. Now, let $a=\{a_i\}_{i=1}^{+\infty}\in A$ be such that $\sum_{i=1}^{n+r}a_i<n$. Then, there exist two indexes, $i_1\in\{1,\dots,n+r\}$ and $i_2>n+r$ such that
    $$0\le a_{i_1}<1\quad\text{and}\quad 0< a_{i_2}\le1.$$
    Let $\varepsilon=\min\{a_{i_2}, 1-a_{i_1}\}>0$ and define the sequence
    $$\tilde{a}_i:=a_i+\varepsilon\delta_{i,i_1}-\varepsilon\delta_{i,i_2}$$
    %$$
    %\tilde{a}_i=
    %\begin{cases}
    %    a_i+\varepsilon & \text{if}\;i=i_1\\
    %    a_i-\varepsilon & \text{if}\;i=i_2\\
    %    a_i & \text{otherwise}.\\
    %\end{cases}$$
    with $\delta_{i,j}$ being the usual Kronecker delta. It is straightforward to see that $\tilde{a}=\{\tilde{a}_i\}_{i=1}^{+\infty}\in A$. We have
    $$l(\tilde{a})=l(a)+(\lambda_{i_1}-\lambda_{i_2})\varepsilon>l(a),$$
    since $\lambda_{i_1}\ge\lambda_{n+r}>\lambda_{n+r+1}\ge\lambda_{i_2}$. Therefore, $a$ is not a maximizer of $l$. This shows that any $a\in A$ maximizing $l$ must satisfy $\sum_{i=1}^{n+r}a_i=n$, as claimed. Finally, we notice $\sum_{i=1}^{n+r}a_i=n\implies a_i=0$ for all $i>n+r$. Thus, for any $a^*\in\mathscr{A}$ maximizing $l$ we have
    $$l(a^*)=\sum_{i=1}^{+\infty}\lambda_ia^*_i = \sum_{i=1}^{n-1}\lambda_i a^*_i + \lambda_n\sum_{i=n}^{n+r}a^*_i,$$
    since $\lambda_j=\lambda_n$ for all $n\le j\le n+r.$
    By leveraging the fact that $\lambda_i\ge\lambda_n$ for $i=1,\dots,n-1$, and by repeating the same argument as before, it is straightforward to see that, without loss of generality we may focus on those maximizers for which $\sum_{i=1}^{n-1}a_i^*=n-1.$ Then, these satisfy $a_i^*=1$ for all $i=1,\dots,n-1$, and, thus $\sum_{i=n}^{n+r}=n-\sum_{i=1}^{n-1}a_i^*=1.$ It follows that $l(a^*)=\sum_{i=1}^{n}\lambda_i.$
\end{proof}

\begin{corollary}
    \label{corollary:Auj}
    Let $(H,\|\cdot\|)$ be a separable Hilbert space. Let $A\in\hilbs(H)$ be given in canonical form as
    $A=\sum_{i=1}^{+\infty}\sigma_i(A)\scl{\cdot}{u_i}v_i.$
    Fix any $n\in\mathbb{N}_+$. For all orthonormal sets $\{\tilde{u}_1,\dots,\tilde{u}_n\}$ one has
    \begin{equation}
    \label{eq:Auj}\sum_{i=1}^{n}\|A\tilde{u}_i\|^2\le\sum_{i=1}^{n}\sigma_i(A)^2.\end{equation}
    Additionally, if $\sigma_n(A)>\sigma_{n+1}(A)$, then the equality can be realized iff $\textnormal{span}(\{\tilde{u}_i\}_{i=1}^{n})=\textnormal{span}(\{u_i\}_{i=1}^{n})$.
\end{corollary}

\begin{proof}
    Let $\tilde U:=\textnormal{span}(\{\tilde{u}_i\}_{i=1}^{n}).$ Let $\tilde P$ be the orthogonal projection from $H$ onto $\tilde U$. Clearly, $\rank(\tilde P)=n$ and, in  particular, $\hsnorm{\tilde P}^2 = n.$ We have
    \begin{align*}
    \sum_{i=1}^{n}\|A\tilde{u}_i\|^2=\sum_{i=1}^{n}\sum_{j=1}^{+\infty}\sigma_j(A)^2|\scl{\tilde u_i}{u_j}|^2=\sum_{j=1}^{\infty}\sigma_j(A)^2\sum_{i=1}^{n}|\scl{\tilde u_i}{u_j}|^2=\sum_{j=1}^{\infty}\sigma_j(A)^2\|\tilde P u_j\|^2.\end{align*}
    Now, let $a_j:=\|\tilde P u_j\|^2.$ By definition, $0\le a_j\le 1$ and $\sum_{j=1}^{+\infty}a_j=\hsnorm{\tilde P}^2=n.$ Then, Eq. \eqref{eq:Auj} follows directly from Lemma \ref{lemma:truncation}. Similarly, if $\sigma_n(A)>\sigma_{n+1}(A)$, Lemma \ref{lemma:truncation} ensures that the equality can be achieved if and only if $\|\tilde P u_j\|=1$ for all $j=1,\dots,n$ and $\|\tilde P u_j\|=0$ for all $j>n.$ Equivalently, $P$ must be the projection onto 
    $\textnormal{span}(\{u_i\}_{i=1}^{n})$. The conclusion follows.
\end{proof}

\begin{lemma}
    \label{lemma:covariance}
     Let $(H,\|\cdot\|)$ be a separable Hilbert space, endowed with a suitable probability measure $\mathbb{P}.$ Let $Z,Z'$ be two $H$-valued random variables such that $\mathbb{E}\|Z\|^2+\mathbb{E}\|Z'\|^2<+\infty$. Let $B,B'\in\trace$ be the (uncentered) covariance operators, $B:u\mapsto \mathbb{E}[\scl{u}{Z}Z]$ and $B':u\mapsto \mathbb{E}[\scl{u}{Z'}Z']$, respectively. Then,
     $$\|B-B'\|_1\le\mathbb{E}^{1/2}\|Z-Z'\|^2\left(\mathbb{E}^{1/2}\|Z\|^2+\mathbb{E}^{1/2}\|Z'\|^2\right).$$     
\end{lemma}
\begin{proof}
    Let $C\in\cmpts$ with $\opnorm{C}\le 1$ and fix any orthonormal basis $\{e_i\}_i\subset H$. We have
    \begin{align*}
    \mathsf{Tr}(C(B-B'))=\sum_{i=1}^{+\infty}\scl{C(B-B')e_i}{e_i}&\le\sum_{i=1}^{+\infty}|\scl{CBe_i}{e_i}-\scl{CB'e_i}{e_i}|
    \\&=\sum_{i=1}^{+\infty}|\mathbb{E}\left[\scl{e_i}{Z}\scl{CZ}{e_i}\right]-\mathbb{E}\left[\scl{e_i}{Z'}\scl{CZ'}{e_i}\right]|.
    \end{align*}
    By applying the triangular inequality, we get
    \begin{align*}
    \mathsf{Tr}(C(B-B'))=\dots=&\sum_{i=1}^{+\infty}|\mathbb{E}\left[\scl{e_i}{Z}\scl{CZ}{e_i}\right]-\mathbb{E}\left[\scl{e_i}{Z}\scl{CZ'}{e_i}\right]|+\\&\sum_{i=1}^{+\infty}|\mathbb{E}\left[\scl{e_i}{Z}\scl{CZ'}{e_i}\right]-\mathbb{E}\left[\scl{e_i}{Z'}\scl{CZ'}{e_i}\right]|.
    \end{align*}
    Notice that the two terms are symmetric in $Z$ and $Z'$, thus, it suffices to study one of the two. For instance, focusing on the first one yields
    \begin{multline*}
    \sum_{i=1}^{+\infty}|\mathbb{E}\left[\scl{e_i}{Z}\scl{CZ}{e_i}\right]-\mathbb{E}\left[\scl{e_i}{Z}\scl{CZ'}{e_i}\right]|=\\=\sum_{i=1}^{+\infty}|\mathbb{E}\left[\scl{e_i}{Z}\scl{C(Z-Z')}{e_i}\right]|
    \le\sum_{i=1}^{+\infty}\mathbb{E}^{1/2}|\scl{e_i}{Z}|^2\mathbb{E}^{1/2}|\scl{C(Z-Z')}{e_i}|^2    \le\\
    \le\left(\sum_{i=1}^{+\infty}\mathbb{E}|\scl{e_i}{Z}|^2\right)^{1/2}\left(\sum_{i=1}^{+\infty}\mathbb{E}|\scl{C(Z-Z')}{e_i}|^2\right)^{1/2}=
    \\
    = \left(\mathbb{E}\sum_{i=1}^{+\infty}|\scl{e_i}{Z}|^2\right)^{1/2}\left(\mathbb{E}\sum_{i=1}^{+\infty}|\scl{C(Z-Z')}{e_i}|^2\right)^{1/2}=
    \\
    = \mathbb{E}^{1/2}\|Z\|^2\cdot\mathbb{E}^{1/2}\|C(Z-Z')\|^2.
    \end{multline*}
    Since $\|C(Z-Z')\|\le\opnorm{C}\cdot\|Z-Z'\|\le\|Z-Z'\|$, it follows that
    $$
    \sum_{i=1}^{+\infty}|\mathbb{E}\left[\scl{e_i}{Z}\scl{CZ}{e_i}\right]-\mathbb{E}\left[\scl{e_i}{Z}\scl{CZ'}{e_i}\right]|\le\dots\le\mathbb{E}^{1/2}\|Z\|^2\cdot\mathbb{E}^{1/2}\|Z-Z'\|^2,
    $$
    and thus $\mathsf{Tr}(C(B-B'))\le \left(\mathbb{E}^{1/2}\|Z\|^2 + \mathbb{E}^{1/2}\|Z'\|^2\right)\mathbb{E}^{1/2}\|Z-Z'\|^2.$ Passing at the supremum over $C$ yields the conclusion.
\end{proof}

\section{Proofs of Section \ref{sec:preliminaries}}
\label{appendix:classical-proofs}

\paragraph{Complementary proof to Lemma \ref{lemma:svd}}
    Let \review{$A\in\cmpts(H_1,H_2)$} be a compact operator \review{between two separable Hilbert spaces}. \review{Then, by classical results, see, e.g., \cite[Theorem VI.17]{reed1980methods} and \cite[Exercise VI.47]{reed1980methods}, one has
    $$A=\sum_{i=n+1}^{+\infty}\sigma_i(A)\scl{\cdot}{u_i}_{H_1}v_i,$$
    for some orthonormal basis $\{u_i\}_i\subset H_1$ and $\{v_i\}_i\in\subset H_2.$ To complete the proof of the Lemma, we aim to prove} 
    that the truncated series $A_n=\sum_{i=1}^{n}\sigma_i(A)\scl{\cdot}{u_i}v_i$ yields the best $n$-rank approximation of $A$ in both the operator norm and the Hilbert-Schmidt norm. Finally, \review{we shall also prove that, in the Hilbert-Schmidt case,} $A_n$ is the unique best approximation of $A$ whenever $\sigma_{n}(A)>\sigma_{n+1}(A)$ strictly. \review{For simplicity, we limit this part of the proof to the case $H_1=H_2$. We shall then write $H$ for the Hilbert space and $\|\cdot\|$ for its norm.}
    
    \review{Starting with} the operator norm, simply notice that for every $x\in H$ one has
    \begin{equation}  
    \|(A-A_n)(x)\|^2=\sum_{i=n+1}^{+\infty}\sigma_i(A)^2|\scl{x}{u_i}|^2\le\sigma_{n+1}(A)^2\sum_{i=n+1}^{+\infty}|\scl{x}{u_i}|^2=\sigma_{n+1}(A)^2\|x\|^2,
  \end{equation}
    due monotonicity of the singular values. It follows that
    $$\opnorm{A-A_n}\le\sigma_{n+1}(A)=\inf\left\{\opnorm{A-L}\;:\;L\in\hilbs,\;\rank(L)\le n\right\},$$
    as claimed.
    Let us now discuss the case of the Hilbert-Schmidt norm. Notice that $L\in\cmpts(H)$ with $\rank(L)\le n$ implies $L\in\hilbs(H)$. In particular, if $A\notin\hilbs(H)$ then \eqref{eq:hs-minimizer} is obvious since $A-L\notin\hilbs(H)$ and thus $\hsnorm{A-L}=+\infty$, coherently with the fact that $\sum_{i>n}\sigma_i(A)^2$ must diverge. Instead, assume that $A\in\hilbs(H)$. We can leverage a Von Neumann type inequality, which states that for any $B\in\hilbs(H)$ one has
    $\textsf{Tr}(AB^*)\le\sum_{i}\sigma_{i}(A)\sigma_{i}(B)$, cf. \cite{dirr2019neumann}. As a direct consequence, see, e.g., \cite[Corollary 7.4.1.3]{horn2012matrix},
    $$\hsnorm{A-B}^2\ge\sum_{i=1}^{+\infty}|\sigma_i(A)-\sigma_i(B)|^2.$$
    We now notice that if $\rank(B)\le n$, then $\sigma_i(B)=0$ for all $i\ge n+1$, as clearly seen by expanding $B$ in its canonical form. It follows that, in this case,
    $$\hsnorm{A-B}^2\ge\sum_{i=1}^{n}|\sigma_i(A)-\sigma_i(B)|^2+\sum_{i>n}\sigma_i(A)^2\ge\sum_{i>n}\sigma_i(A)^2=\hsnorm{A-A_n}^2,$$
    as claimed. The above also shows that any $n$-rank operator $B$ minimizing $\hsnorm{A-B}$ must satisfy $\sigma_i(B)=\sigma_i(A)$ for all $i=1,\dots,n.$ In particular, $B$ and $A_n$ must share the same singular values. In light of this fact, let $B=\sum_{i=1}^{n}\sigma_i(A)\scl{\cdot}{\tilde u_i}\tilde v_i$. We have
    $$\hsnorm{A-B}^2=\hsnorm{A}^2+\hsnorm{B}^2-2\textsf{Tr}(B^*A)=\sum_{i=1}^{+\infty}\sigma_i(A)^2+\sum_{i=1}^{n}\sigma_i(A)^2-2\textsf{Tr}(B^*A).$$
    In particular, $B$ must be such that $\textsf{Tr}(B^*A)$ is maximized. To this end, let $\{\tilde u_{j}\}_{j=n+1}^{+\infty}$ be an orthonormal basis for $\tilde U^\top$, where $\tilde U:=\text{span}(\{\tilde u_j\}_{j=1}^{m})$, so that $\{\tilde u_j\}_{j=1}^{+\infty}$ forms an orthonormal basis of $H.$ Since $B\tilde u_j=0$ for all $j>n$, we have,
    $$\textsf{Tr}(B^*A)=\sum_{i=1}^{n}\scl{A\tilde u_i}{B\tilde u_i}=\sum_{i=1}^{n}\sigma_i(A)\scl{A\tilde u_i}{\tilde v_i}\le\sqrt{\sum_{i=1}^{n}\sigma_{i}^2(A)}\sqrt{\sum_{i=1}^{n}|\scl{A\tilde u_i}{\tilde v_i}|^2}.$$
    It follows that
    $$\textsf{Tr}(B^*A)\le \sqrt{\sum_{i=1}^{n}\sigma_{i}^2(A)}\le \sqrt{\sum_{i=1}^{n}\|A\tilde u_i\|^2}\le \sum_{i=1}^{n}\sigma_{i}^2(A),$$
    the last inequality coming from Corollary \ref{corollary:Auj}. On the other hand, the upper-bound is reached whenever $B=A_n$, thus, if $B$ is actually a maximizer of $\textsf{Tr}(B^*A)$, the above needs to be an equality. In turn, this implies that $\tilde U= U = \text{span}(\{u_i\}_{i=1}^{n})$, once again by Corollary \ref{corollary:Auj}. In particular, since $A$ and $A_n$ coincide on $U$, we may re-write $\textsf{Tr}(B^*A)$ as
    $$\textsf{Tr}(B^*A)=\dots=\sum_{i=1}^{n}\sigma_i(A)\scl{A\tilde u_i}{\tilde v_i}=\sum_{i=1}^{n}\sigma_i(A_n)\scl{A_n\tilde u_i}{\tilde v_i}=\dots=\textsf{Tr}(B^*A_n).$$
    On the other hand, due minimality (resp. maximality) it must be
    $$\textsf{Tr}(B^*A_n)=\textsf{Tr}(B^*A)=\sum_{i=1}^{n}\sigma_i(A)^2=\hsnorm{A_n}^2$$
    Now recall that $\hsnorm{B}=\hsnorm{A_n}=:\rho$. Since $\hilbs$ is a Hilbert space, due uniform convexity, there is a unique element in $\{C\in\hilbs\;\text{with}\;\hsnorm{C}\le \rho\}$ maximizing $\scl{C}{A_n}_{\textnormal{HS}}=\textsf{Tr}(C^*A_n)$, which is precisely $A_n.$ It follows that $B=A_n$, as claimed.\qed
    %
    %By repeating the same ideas on the adjoints (in fact, $\hsnorm{A^*-B^*}=\hsnorm{A-B}$, thus $B^*$ is optimal for $A^*$), it is straightforward to conclude that the same also holds for the left singular vectors, i.e.
    %$$\text{span}(\{v_1,\dots, v_n\})=\text{span}(\{\tilde v_1,\dots, \tilde v_n\}).$$
    %In particular, there exists two unitary operators $R_1,R_2:H\to H$ such that $\tilde{u}_i=R_1(u_i)$ and $\tilde{v}_i=R_2(v_i)$ for all $i=1,\dots,n$. In particular, since $R_1(U)=\tilde{U}=U$

\paragraph{Proof of Lemma \ref{lemma:pod}}
Let $(H,\|\cdot\|)$ be a separable Hilbert space and let $X$ be a square-integrable $H$-valued random variable with a given probability law $\mathbb{P}$. Let $B:H\to H$ be the linear operator
    $B:u\mapsto \mathbb{E}\left[\scl{u}{X}X\right].$ We prove the following.

\begin{itemize}
    \item [i)] $B\in\trace$ in a symmetric positive semidefinite trace class operator.

    \begin{proof}
    We notice that, for all $u,u'\in H$, we have
    $$\langle B(u), u'\rangle = \langle\expe[\langle X, u\rangle X],u'\rangle=\expe[\langle X, u\rangle \langle X,u'\rangle] = \langle u, B(u')\rangle.$$
    In particular, for $u'=u$, $\langle B(u), u\rangle=\expe|\scl{X}{u}|^2\ge0.$ Thus, having fized any orthonormal basis $\{e_i\}_i\subset H$, we may compute $\|B\|_1$ as
    \begin{equation*}
    \sum_{i=1}^{+\infty}\langle B(e_{i}),e_{i}\rangle=\sum_{i=1}^{+\infty}\expe[\langle X, e_{i}\rangle^{2}]=\expe\left[\sum_{i=1}^{+\infty}\langle X, e_{i}\rangle^{2}\right]=\expe\|X\|^{2}<+\infty.\tag*{\qedhere}\end{equation*}\end{proof}
    \item [ii)]     
    There exists a sequence of (scalar) random variables $\{\eta_{i}\}_{i=1}^{+\infty}$ with $\expe[\eta_i\eta_j]=\delta_{i,j}$ such that
        $X=\sum_{i=1}^{+\infty}\sqrt{\lambda_i}\eta_iv_i$
        $\mathbb{P}$-almost surely, where $\lambda_1\ge\lambda_2\ge\dots\ge0$ and $v_i\in H$ are the eigenvalues and eigenvectors of $B$, respectively.
    \begin{proof}
    Let $\eta_{i}:=\langle X,v_{i}\rangle/\sqrt{\lambda_{i}}.$
    It is straightforward to see that for all $i,j\in\mathbb{N}$ one has 
    \begin{align*}
    \expe[\eta_{i}\eta_{j}]=\frac{1}{\sqrt{\lambda_{i}\lambda_{j}}}\expe[\langle X, v_{i}\rangle\langle X, v_{j}\rangle]=\frac{1}{\sqrt{\lambda_{i}\lambda_{j}}}\langle B(v_{i}),v_{j}\rangle=\frac{\lambda_{i}}{\sqrt{\lambda_{i}\lambda_{i}}}\langle v_{i},v_{j}\rangle=\delta_{i,j}.
    \end{align*}
    Finally, it is clear that $X=\sum_{i=1}^{+\infty}\sqrt{\lambda_{i}}\eta_{i}v_{i}$ by definition.\end{proof}
    \item[iii)]
    For every orthogonal projection $P:H\to H$ one has $\mathbb{E}\|X-PX\|^2=\mathbb{E}\|X\|^2-\sum_{i=1}^{+\infty}\lambda_i\|Pv_i\|^2.$
    \begin{proof}
    Let $P:H\to H$ be an orthogonal projection and let $V:=P(H)$. Let $\{w_j\}_{j\in J}$ be an orthonormal basis of $V$, be it finite or infinite depending on the dimension of $V\subseteq H.$ Due orthogonality,
    $$\mathbb{E}\|X-PX\|^2=\mathbb{E}\|X\|^2-\mathbb{E}\|PX\|^2.$$
    Expanding the second term, due $\mathbb{E}$-orthonormality of the $\eta_{i}$ and orthonormality of the $w_j$'s, reads,
        \begin{multline*}
        \mathbb{E}\|PX_\gparam\|^2=\mathbb{E}\left[\sum_{j\in J}|\scl{w_j}{X}|^2\right]=\sum_{j\in J}\mathbb{E}\left|\sum_{i=1}^{\infty}\eta_{i}\sqrt{\lambda_i}\scl{w_j}{v_i}\right|^2=\\=\sum_{j\in J}\sum_{i=1}^{+\infty}\lambda_i|\scl{w_j}{v_i}|^2=\sum_{i=1}^{+\infty}\lambda_i\sum_{j\in J}|\scl{w_j}{v_i}|^2=\sum_{i=1}^{+\infty}\lambda_i\|Pv_i\|^2.\tag*{\qedhere}\end{multline*}
    \end{proof}
    \item [iv)] For every $n\in\mathbb{N}_+$ the random variable $X_n:=\sum_{i=1}^{n}\sqrt{\lambda_i}\eta_iv_i$ satisfies
        $\mathbb{E}\|X-X_n\|^2=\inf_{Z\in Q_n}\mathbb{E}\|X-Z\|^2,$
        where $Q_n=\{Z\in L^{2}_H\;:\;\exists V\subseteq H,\;\dim(V)\le n,\;Z\in V\;\mathbb{P}\text{-almost surely}\}.$
    \begin{proof}
    Fix any $Z\in Q_n.$ Let $V\subseteq H$ be a subspace of dimension $n$ such that $Z\in V$ $\mathbb{P}$-almost surely. Let $P:H\to V$ denote the orthogonal projection onto $V$. Due optimality of orthogonal projections, we have
    $$\mathbb{E}\|X-Z\|^2\ge\mathbb{E}\|X-PX\|^2=\mathbb{E}\|X\|^2-\sum_{i=1}^{+\infty}\lambda_i\|Pv_i\|^2,$$
    the equality following from point (iii). Now, let $a_i:=\|Pv_i\|^2$, and let $\{w_j\}_{j=1}^{n}$ be an orthonormal basis for $V$. Without loss of generality, extend the latter to a complete orthonormal basis, $\{w_j\}_{j=1}^{+\infty}$, spanning the whole $H$. We notice that $0\le a_i\le 1$ and $$\sum_{i=1}^{+\infty}a_i=\|P\|_{\textnormal{HS}}^2=\sum_{j=1}^{+\infty}\|Pw_j\|^2=\sum_{j=1}^{n}\|w_j\|^2=n.$$
    Therefore, as a direct consequence of Lemma \ref{lemma:truncation}, 
    $\mathbb{E}\|X-Z\|^2\ge\dots\ge\mathbb{E}\|X\|^2-\sum_{i=1}^{n}\lambda_i=\sum_{i=1}^{+\infty}\lambda_i-\sum_{i=1}^{n}\lambda_i=\sum_{i=n+1}^{+\infty}\lambda_i.$
    Since $\mathbb{E}\|X-X_n\|^2$ is easily shown to equal $\sum_{i=n+1}^{+\infty}\lambda_i$, the conclusion follows.
    \end{proof}
\end{itemize}

\paragraph{Proof of Theorem \ref{theorem:argmin}}
\review{Let $(\gspace,d_{\gspace})$ and $(C,d_{C})$ be Polish spaces, $C$ compact, and let $J:\gspace\times C\to \mathbb{R}$ be: \emph{(i)} lower semi-continuous in $(\gparam,c)$; \emph{(ii)} marginally continuous in $\gparam$. We aim to prove the existence of a Borel measurable map $c_*:\gspace\to C$ such that
    $$J(\gparam,c_*(\gparam))=\min_{c\in C}J(\gparam,c)\quad\quad\forall \gparam\in \gspace.$$
We will then show that $c_*$ is continuous if the parametrized minimization problem admits a unique solution for every $\gparam\in\gspace.$}

\review{Starting with the measurability, let us define the map $F:\gspace\to 2^{C}$ 
$$F:\gparam\mapsto\left\{c\in C\;\text{such that}\;J(\gparam,c)=\min_{c'\in C}J(\gparam,c')\right\},$$
mapping every $\gparam\in\gspace$ onto a suitable subset of $C$. We shall prove that
\begin{itemize}
    \item[(i)] for every $\gparam\in \gspace$, the set $F(\gparam)\subseteq C$ is closed and not empty;
    \item[(ii)] for every open set $A\subseteq C$, the set
$$\mathcal{S}_{A}:=\{\gparam\in \gspace\;:\;F(\gparam)\cap A\neq\emptyset\}$$
is Borel measurable.
\end{itemize}
In order to prove the above, for every $\gparam\in \gspace$, let us denote by $j_\gparam$ the map from $C\to\mathbb{R}$ defined as $j_\gparam(c):=J(\gparam,c).$ By hypothesis, $j_\gparam$ is lower-semicontinuous as a function of $c\in C$. Then, (i) is trivial. In fact, it is well known that every lower-semicontinuous map defined over a compact metric space admits a minimum. In particular, if $\alpha_{\gparam}:=\min_{c'\in C}J(\gparam,c'),$ then
$$F(\gparam)=j_\gparam^{-1}\left(\left\{\alpha_x\right\}\right)=
j_\gparam^{-1}\left(-\infty,\;\alpha_\gparam\right],$$
is closed (by lower-semicontinuity of $j_\gparam$) and nonempty (by minimality). As for the second statement, instead, let us first consider the case in which $A\subset C$ is compact. We aim to prove that $\mathcal{S}_A$ is closed, and thus Borel measurable. To this end, let $\{\gparam_{n}\}_{n}\subseteq \mathcal{S}_{A}$ be a sequence converging to some $\gparam\in \gspace$. By definition of $\mathcal{S}_{A}$, for each $\gparam_{n}$ there exists some  $c_{n}\in A$ such that $c_{n}\in F(\gparam_{n})$, i.e. for which $J(\gparam_{n},c_{n})=\min_{c'} J(\gparam_{n},c')$. Since $A$ is compact, up to passing to a subsequence, there exists some $c\in A$ such that $c_{n}\to c$. 
Let now $\tilde{c}\in C$ be a minimizer for $\gparam$, i.e. a suitable element for which $J(\gparam,\tilde{c})=\min_{c'\in C}J(\gparam,c')$. Since $J$ is lower semi-continuous but also continuous in its first argument,
\begin{equation*}    J(\gparam,c)\le \liminf_{n\to+\infty}J(\gparam_n,c_{n})=\liminf_{n\to+\infty}\min_{c' \in C}J(\gparam_{n},c')\le\liminf_{n\to+\infty}J(\gparam_{n},\tilde{c})=J(\gparam,\tilde{c}),
\end{equation*}
implying that $c$ is also a minimizer for $\gparam$. It follows that $c\in A\cap F(\gparam)$ and thus $\gparam\in \mathcal{S}_{A}$. In particular, $\mathcal{S}_{A}$ is closed. Finally, to see that (ii) holds notice that every open set $A\subseteq C$ can be written as the countable union of compact sets $A=\cup_{n\in\mathbb{N}}A_n$. Since $\mathcal{S}_A=\cup_{n\in\mathbb{N}}\mathcal{S}_{A_n}$, the conclusion follows.}

\review{Since both (i) and (ii) hold true, we notice that $F$ is in fact a \emph{measurable set-valued map}, see, e.g., \cite[Definition 8.1.1]{aubin2009set}. In particular, we are in the position to invoke the celebrated Kuratowski–Ryll-Nardzewski selection theorem, which ensures the existence of a measurable map $c_*:\gspace\to C$ such that $c_*(\gparam)\in F(\gparam)$, i.e. $J(\gparam,c_*(\gparam))=\min_{c\in C}J(\gparam,c)$: see, e.g., \cite[Theorem 8.1.3]{aubin2009set}. The first statement in Theorem \ref{theorem:argmin} is thus proven.}
\\\\
\review{Concerning the continuity, instead, consider the limit case $A=C$ in the previous argument. Then, $\mathcal{S}_A=\gspace$. Since each $\gparam$ now admits a unique minimizer $c_*(\gparam)\in C$, the above argument shows that for every sequence $\{\gparam_n\}_n\subseteq \gspace$ converging to some $\gparam\in\gspace$, the corresponding sequence of minimizers $\{c_*(\gparam_n)\}_{n}$ admits a subsequence $\{c_*(\gparam_{n_k})\}_k$ converging to a suitable minimizer of $J(\gparam,\cdot)$. Then, by uniqueness, it must be $c_*(\gparam_{n_k})\to c_*(\gparam).$ In other words, for every convergent sequence $\gparam_n\to\gparam$ in $\gspace$, there exists a subsequence such that $c_*(\gparam_{n_k})\to c_*(\gparam).$ It follows that $c_*$ is continuous.} 

\review{To see this, notice that $c_*^{-1}(K)$ is closed whenever $K\subseteq C$ is closed. In fact, if $\gparam_n\to\gparam$ with $\{\gparam_n\}_n\subseteq c_*^{-1}(K)$, then there exists a subsequence $\{c_*(\gparam_{n_k})\}_k\subseteq K$ with $c_*(\gparam_{n_k})\to c_*(\gparam)$. Since $K$ is closed, we have $c_*(\gparam)\in K$, meaning that $\gparam\in c_*^{-1}(K)$. Thus, $c_*^{-1}(K)$ is closed. Since $K$ was arbitrary, this shows that $c_*$ is continuous.}\qed

\end{appendices}

\end{document}